\pdfoutput=1
\documentclass[12pt, reqno]{amsart}
\usepackage{graphicx}
\usepackage{tikz}
\usepackage{caption}
\usepackage[margin=1in]{geometry}
\usepackage[labelfont=rm]{subcaption}
\usepackage{blindtext}
\numberwithin{equation}{section}
\usepackage[T1]{fontenc}
\usepackage{soul, color}
\usepackage{amsmath}
\usepackage{amssymb}
\usepackage{enumitem}
\usepackage{hyperref}

\usepackage{amsthm}
\usepackage{setspace}

\usepackage{mathtools}
\usepackage{commath}
\usepackage[numbers,sort]{natbib}
\usepackage[bottom]{footmisc}
\newcommand{\pr}[1]{#1^{\prime}}
\newcommand{\x}{\mathbf{x}}
\newcommand{\iid}{\mathrm{i.i.d}}
\newcommand{\Var}{\mathrm{Var}}
\newcommand{\Cov}{\mathrm{Cov}}

\newcommand{\R}{\mathbb{R}}
\newcommand{\one}{{\bf 1}}
\def\P{\mathbb{P}}
\def\E{\mathbb{E}}
\newcommand{\Pn}{\mathcal{P}_n}

\newcommand{\X}{\mathcal{X}}
\newcommand{\Y}{\mathcal{Y}}

\newcommand{\Nat}{\mathbb{N}}
\newcommand{\y}{\mathbf{y}}

\newcommand{\VR}{\mathcal R}
\newcommand{\ind}[1]{\mathbf{1}\big\{#1\big\}}

\newcommand{\geom}[1]{\mathrm{geom}(#1)}
\newcommand{\id}{\text{id}}
\newcommand{\card}[1]{|#1|}

\newcommand{\ed}[1]{{\color{red}#1}}
\setstcolor{red}

\newcommand{\remove}[1]{}

\newcommand{\xn}{\bar{\chi}_n}
\newcommand{\xnA}{\bar{\chi}_{n,A}}

\newcommand{\barxi}{\bar{\xi}}

\newcommand{\bx}{{\bf x}}
\newcommand{\by}{{\bf y}}

\newcommand{\bb}{{\bf b}}

\newcommand{\B}{\mathcal B}

\newcommand{\kokt}{k_1+k_2}

\theoremstyle{plain}
\newtheorem{thm}{Theorem}[section]
\newtheorem{lem}[thm]{Lemma}
\newtheorem{prop}[thm]{Proposition}

\theoremstyle{defn}
\newtheorem{defn}[thm]{Definition}
\newtheorem{rem}[thm]{Remark}
\newtheorem{ex}[thm]{Example}

\begin{document}

\title[Limit theorems for the Euler characteristic process]{Functional limit theorems for the Euler characteristic process in the critical regime} 
\author{Andrew M. Thomas and Takashi Owada}

\address{Department of Statistics\\
Purdue University \\
IN, 47907, USA}
\email{thoma186@purdue.edu \\ owada@purdue.edu}

\thanks{This research is partially supported by the NSF grant DMS-1811428.}


\begin{abstract}
This study presents functional limit theorems for the Euler characteristic of Vietoris-Rips complexes. The points are drawn from a non-homogeneous Poisson process on $\R^d$, and the connectivity radius governing the formation of simplices is taken as a function of time parameter $t$, which allows us to treat the Euler characteristic as a stochastic process. The setting in which this takes place is that of the critical regime, in which the simplicial complexes are highly connected and have non-trivial topology. 
We establish two ``functional-level" limit theorems, a strong law of large numbers and a central limit theorem for the appropriately normalized Euler characteristic process.
%This is due to the fact that the radius of the balls $r_n(t)$, has the behavior $n^{1/d}r_n(t) \to t$. The Poincar{\'e} inequality of \cite{Last2011} allows us to establish limiting covariance results of $\chi^{\ell}_n(t)$, which in turn allow us to prove finite-dimensional weak convergence $n^{-1/2}(\chi^{\ell}_n(t) - \E[\chi^{\ell}_n(t)])$ to a centered multivariate normal. The aforementioned finite-dimensional convergence and a proof of tightness, allow us to prove a functional central limit theorem for $\chi^{\ell}_n(t)$, when $\ell \equiv 1$ induces a Vietoris-Rips complex. 
\end{abstract}

\maketitle

\section{Introduction}

 The Euler characteristic is one of the oldest and simplest topological summaries. It is at once local and global, combinatorial and topological, owing to its representation as either the alternating sum of Betti numbers of a topological space, or the alternating sum of simplices in its triangulation. Beyond its theoretical beauty, the Euler characteristic has recently made its way into the field of applied mathematics, notably topological data analysis (TDA). For instance, the Euler characteristic of sublevel (or superlevel) sets of random fields have found broad applications \cite{Adler2008, Crawford2016}. %among them, astrophysics \cite{gott1986, vogeley1994}, medical imaging \cite{Taylor2007, worsley1992}, and   image classification \cite{Richardson2014}. 
In TDA, the technique of capturing the dynamic evolution of topology is generally studied in \emph{persistent homology}---see \cite{Carlsson2009} for a good introduction. Persistent homology originated in computational topology \cite{edelsbrunner2010} and has received much attention as a useful machinery for exploring the manner in which topological holes appear and/or disappear in a filtered topological space. The primary objective of the current study is to associate the Euler characteristic with some filtered topological space by treating it as a stochastic process in time parameter $t$. 
 
Due to recent rapid development of TDA in conjunction with probability theory, there has been a growing interest in the study of random geometric complexes. We focus on the \emph{Vietoris-Rips complex} \cite{kahle2013limit, Kahle2011, owada:2019}, due to its ease of application, especially those in computational topology; though much research has also been done on the \emph{{\v C}ech complex} \cite{kahle2013limit, Kahle2011, owathom, Yogeshwaran2017, decreusefond2014simplicial, bobrowski2011distance, Bobrowski2015, Goel2018}, or the notion of generalizing both types of complex \cite{Duy2016}.
An elegant survey of progress these areas can be found in  \cite{Bobrowski2018}. 
These studies are mostly concerned with the asymptotic behavior of topological invariants such as the Euler characteristic and Betti numbers. Among them, \cite{decreusefond2014simplicial} derived a concentration inequality for the Euler characteristic built over a \v{C}ech complex on a $d$-dimensional torus, as well as its asymptotic mean and variance; and \cite{hug:last:schulte:2016} established a multivariate central limit theorem for the intrinsic volumes, including the Euler characteristic. Furthermore, \cite{schneider:weil:2008} proved ergodic theorems for the Euler characteristic over a stationary and ergodic point process. 

Most of the studies cited in the last paragraph start with either an iid random sample $\mathcal X_n=\{ X_1,\dots,X_n \}$ or a Poisson point process $\Pn=\{ X_1,\dots,X_{N_n} \}$, where $N_n$ is a Poisson random variable with mean $n$, independent of $(X_i)_i$. Subsequently, we will consider a simple Boolean model of the union of balls centered around $\mathcal X_n$ or $\Pn$ with a sequence of non-random radii $s_n\to 0$, $n\to\infty$.  Then, the behavior of topological invariants based on the Boolean model can be split up into several distinct regimes. When $ns_n^d \to 0$, $n\to\infty$, we have what is called the \emph{sparse} (or subcritical) regime, in which there occur many small  connected components. If $n s_n^d \to \infty$ as $n\to\infty$,  we have the \emph{dense} (or supercritical) regime, which is characterized by a large connected component with few topological holes as a result of a slower decay rate of $s_n$. An intermediate case for which $n s_n^d$ converges to a positive and finite constant, is called the \emph{critical} regime, in which the stochastic features of a geometric complex are less assured, and are arguably more interesting, due to the emergence of highly connected components with non-trivial topologies. The present study focuses exclusively on the critical regime. This is because the behaviors of the Euler characteristic in other regimes, e.g., sparse and dense regimes, are considerably trivial. For example, in the dense regime, the Euler characteristic is asymptotic to $1$ (see \cite{bobrowski2011distance}).
 
Within the context of geometric complexes---such as the \v{C}ech and Vietoris-Rips complexes---few attempts have been made thus far at deriving limit theorems on the functional level for topological invariants (with a few exceptions---see \cite{owathom, owada:2019, biscio2019}). From the viewpoint of persistent homology, such functional information is crucial for the understanding of topological invariants in a filtered topological space. 
With this in mind, the current study proceeds to establish functional limit theorems for the Euler characteristic defined as a stochastic process. 
More specifically, we shall prove a functional strong law of large numbers and a functional central limit theorem in the space $D[0,\infty)$ of right continuous functions with left limits. Our results are the first functional limit theorems in the literature for a topological invariant under the critical regime that have neither time/radius restrictions nor restriction on the number/size of components in the underlying simplicial complex.
The primary benefit in our results lies in information obtainable about topological changes as time parameter $t$ varies. For example, if we let $\chi_n(t)$ be the Euler characteristic considered as a stochastic process, then as consequences of our main theorems, one can capture the limiting behavior of various useful functions of the Euler characteristic process via the continuous mapping theorem. We elaborate on these at the end of Section \ref{results}. Other potential applications can be found in Chapter 14 of \cite{Billingsley} and \cite{whitt:2002}. 

In section \ref{preliminaries} we discuss all the topological background necessary for the paper. In section \ref{results} we discuss our main results: the functional strong law of large numbers and functional central limit theorems for the Euler characteristic process in the critical regime. All of the proofs in the paper are collected in Section \ref{proofs}.

\section{Preliminaries} \label{preliminaries}

\subsection{Topology}  

The main concept in the present paper is the Euler characteristic. Before introducing it we begin with the notions of a \emph{simplex} and an \emph{(abstract) simplicial complex}. Let $\Nat$, $\Nat_0$  be the positive and non-negative integers respectively, and $B(x,r)$ be the closed ball centered at $x$ with radius $r \ge 0$. 

\begin{defn} \label{d:abs_simp}
Let $\X$ be a finite set. An \emph{abstract simplicial complex} $\mathcal{K}$ is a collection of non-empty subsets of $\X$ which satisfy
\begin{enumerate}
\item All singleton subsets of $\X$ are in $\mathcal{K}$,
\item If $\sigma \in \mathcal{K}$ and $\tau \subset \sigma$, then $\tau \in \mathcal{K}$. 
\end{enumerate}
If $\sigma \in \mathcal{K}$ and $\card{\sigma} = k + 1$, with $k \in \Nat_0$, then $\sigma$ is said to have dimension $k$ and is called a \textit{$k$-simplex} in $\mathcal{K}$. The dimension of $\mathcal{K}$ is the dimension of the largest simplex in $\mathcal{K}$.
\end{defn}
It can be shown (cf.~\cite{edelsbrunner2010}) that every abstract simplicial complex $\mathcal{K}$ of dimension $d$ can be embedded into $\R^{2d+1}$. The image of such an embedding, denoted $\geom{\mathcal{K}}$, is called the \emph{geometric realization} of $\mathcal{K}$. A topological space $Y$ is said to be \emph{triangulable} if there exists a simplicial complex $\mathcal{K}$ together with a homeomorphism between $Y$ and $\geom{\mathcal{K}}$. %If $Y$ is triangulable, $\mathcal{K}$ is said to be a simplicial complex structure on $Y$. 
We now define the Euler characteristic. 

\begin{defn} \label{d:euler}
Take $\mathcal{K}$ to be a simplicial complex and let $S_k(\mathcal{K})$ be the number of $k$-simplices in $\mathcal{K}$. Then the \emph{Euler characteristic} of $\mathcal{K}$ is defined as 
$$ 
\chi(\mathcal{K}) := \sum_{k=0}^{\infty} (-1)^k S_k(\mathcal{K}).
$$
\end{defn}

If $Y$ is a triangulable topological space with an associated simplicial complex $\mathcal K$, then we have $\chi(Y) = \chi(\mathcal{K})$, and $\chi(Y)$ is independent of the triangulation (see Theorem 2.44 in \cite{hatcher}). Therefore, the Euler characteristic is a topological invariant (and in fact a homotopy invariant). 

\begin{figure}[!t]
\begin{center}
\includegraphics[width=6.5in]{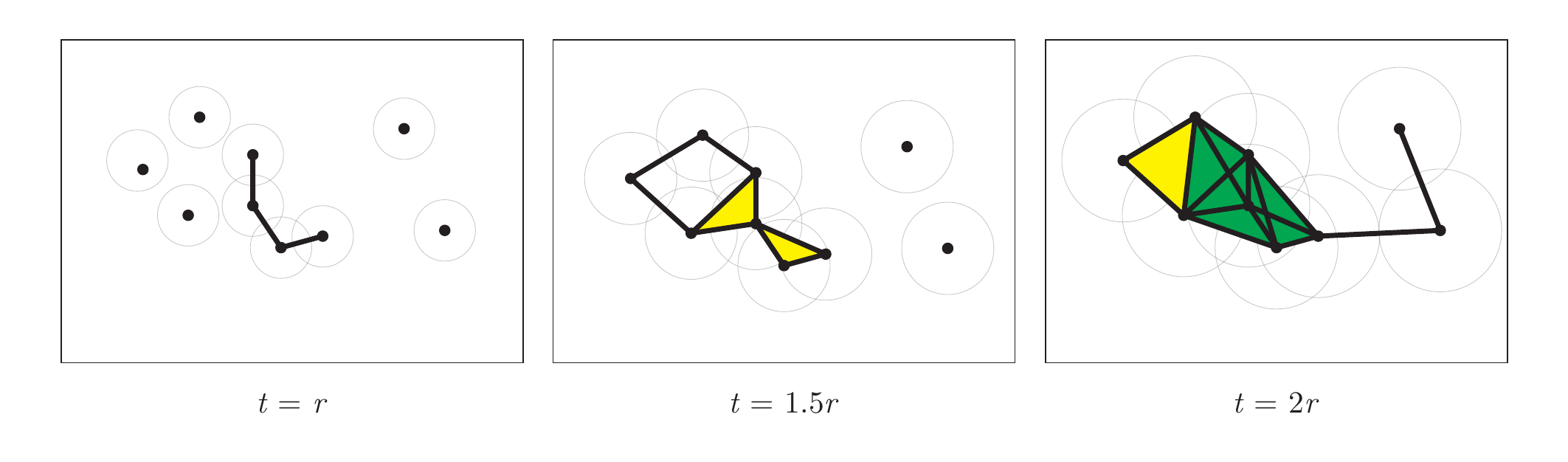}
\caption{A family of Vietoris-Rips complexes.}
\label{f:vietoris_rips}
\end{center}
\end{figure}

Our setting for this study is always in $\R^d$, so we may take $\X$, $\Y$ to be arbitrary finite subsets of $\R^d$. To conclude this section, we will now define the Vietoris-Rips complex: the aforementioned simplicial complex that allows us to get a topological, as well as combinatorial, structure from our data $\X$. A family of Vietoris-Rips complexes $( \VR(\X, t), \, t \ge 0 )$ for points in $\R^2$ can be seen in Figure~\ref{f:vietoris_rips}; yellow represents a 2-simplex and green represents a 3-simplex, which cannot be embedded in $\R^2$.
\begin{defn}
Let $\X = \{x_1, \dots, x_n\}$ be a finite subset of $\R^d$ and $t\ge0$. The \emph{Vietoris-Rips complex} $\VR(\X,t)$ is the (abstract) simplicial complex where 
\begin{enumerate}
\item All singleton subsets of $\X$ are in $\VR(\X,t)$,
\item A $k$-simplex $\sigma = \{x_{i_0}, \dots, x_{i_k}\}$ is in $\VR(\X,t)$ if 
\[
B(x_{i_j}, t) \cap B(x_{i_\ell}, t) \neq \emptyset
\]
for all $0 \leq j < \ell \leq k$. 
\end{enumerate}
\end{defn}

\subsection{Tools}
Throughout, we denote $\Pn$ to be a Poisson point process on $\R^d$ with intensity measure $n\int_Af(x)\dif x$, where $A$ is a Borel subset of $\R^d$, and $f$ is a probability density function. Writing $m$ for Lebesgue measure on $\R^d$, we assume that $f$ is bounded almost everywhere, i.e., $\lVert f \rVert_{\infty} := \inf \big\{a \in \R: m\big(f^{-1}(a, \infty)\big) = 0\big\}<\infty$. 

%\remove{
%$\lambda_n$, where $\lambda_n$ defined on the Borel sets in $\R^d$ which we denote $\mathcal{B}(\R^d)$. Furthermore, $\lambda_n$ is defined by $\lambda_n(A) = n\int_A f(x) \dif{x}$ for $A \in \mathcal{B}(\R^d)$. For $y \in \R^d$ and $r \geq 0$, denote the closed ball $B(y, r) = \{x \in \R^d: \norm{x-y} \leq r\}$. Let $m$ be Lebesgue measure on $\R^d$ and $\theta_d := m(B(0,1))$. Define the essential supremum of $f$ as $\lVert f \rVert_{\infty} := \inf \{a \in \R: m(f^{-1}(a, \infty)) = 0\}$. The only assumptions we make on $f$ are that $\int_{\R^d} f(x) \dif{x} = 1$ and $\lVert f \rVert < \infty$, so that $f$ is bounded almost everywhere. 
%}

For two finite subsets $\Y\subset \X$ of $\R^d$ with $|\Y|=k+1$, and $t\ge0$, we define 
\begin{align}
h_t^{k} (\Y) &:=  \one \big\{ \Y \text{ forms a } k\text{-simplex in } \VR (\X,t)  \big\}  = \prod_{x, y \in \Y, \,  x \neq y} \one \Big\{ B(x, t) \cap B(y, t) \neq \emptyset \Big\}. \label{e:def.indicator}
\end{align}
In the below we present obvious, but highly useful properties of this indicator function. First, it is translation and scale invariant: for any $c>0$, $x\in\R^d$, and $y_0, \dots, y_k \in \R^d$, 
\[
h^{k}_t(cy_0 + x, \dots, cy_k + x) = h^{k}_{t/c}(y_0, \dots, y_k).
\]
Furthermore, for any fixed $y_i \in \R^d$, $i=0,\dots,k$, it is non-decreasing in $t$, i.e., 
\begin{equation}  \label{e:monotonicity.indicator}
h_s^{k}(y_0,\dots,y_k) \le h_t^{k}(y_0, \dots, y_k), \ \ \ 0 \le s \le t. 
\end{equation}
%%%%%%%%%%%%%%%%%%%%%%%%%%%%%%%%%%%%%%%%%%%%%%%%%%%%%%%%%%
%%%%%%%%%%%%%%%%%%%%%%%%%REMOVE%%%%%%%%%%%%%%%%%%%%%%%%%%%%
%%%%%%%%%%%%%%%%%%%%%%%%%%%%%%%%%%%%%%%%%%%%%%%%%%%%%%%%%%
\remove{Finally, we have that
\begin{equation}
h_t^{k, \ell}(y_0, \dots, y_k) \leq h_t^{k, \id \wedge \one}(y_0, \dots, y_k), \ \ t \ge 0, \   y_i \in \R^d, \ i=0,\dots,k.  \label{e:hkl_prop_2}
\end{equation}}

Using \eqref{e:def.indicator}, we can define $k$-simplex counts by 
$
S_k(\X, t) := \sum_{\Y \subset \X} h_t^k (\Y). 
$
As declared in the Introduction, we shall exclusively focus on the critical regime, so that $ns_n^d\to1$, $n\to \infty$. Finally, in order to formulate the Euler characteristic as a stochastic process, let $r_n(t):=s_nt$ and define 
\begin{equation}  \label{e:Euler.characteristic.proc}
\chi_n(t) := \sum_{k=0}^\infty (-1)^k S_k \big( \Pn, r_n(t) \big) = \sum_{k=0}^\infty (-1)^k \sum_{\Y \subset \Pn} h_{r_n(t)}^k (\Y), \ \ t \ge 0. 
\end{equation}
Notice that \eqref{e:Euler.characteristic.proc} is almost surely a finite sum because the cardinality of $\Pn$, denoted as $|\Pn|$, is finite a.s.~and $S_k\big( \Pn, r_n(t) \big)\equiv 0$ for all $k \ge |\Pn|$. Furthermore, for a Borel subset $A$ of $\R^d$, define  a restriction of the Euler characteristic to $A$ by
\begin{equation}  \label{e:left.most.Euler.characteristic}
\chi_{n,A}(t) := \sum_{k=0}^\infty (-1)^k \sum_{\Y \subset \Pn} h_{r_n(t)}^k (\Y) \one \big\{  \text{LMP}(\Y)\in A \big\}, 
\end{equation}
where $\text{LMP}(\Y)$ represents the left-most point of $\Y$, i.e., the least point with respect to lexicographic order in $\R^d$. This restriction is useful for proving finite-dimensional convergence in the case when $A$ is bounded. When $A$ is bounded we get a finite number of random variables for the dependency graph, so that we may use Stein's method for normal approximation. See Section \ref{sec:fidi.dim.proof} for more details. Clearly, $\chi_{n,\R^d}(t) = \chi_n(t)$. 

\section{Main results} \label{results}

The first contribution of the present paper is the functional strong law of large numbers (FSSLN) for $\chi_n$ in the space $D[0,\infty)$ of right continuous functions with left limits. More precisely, almost sure convergence of $\chi_n/n$ to the limiting mean will be established in terms of the uniform metric. Our proof techniques rely on the Borel-Cantelli lemma to prove a strong law of large numbers for each fixed $t$ and we then extend this to the functional case. As for the method of proofs in other studies, \cite{penr} and \cite{Yogeshwaran2017} have established concentration inequalities that can lead to the desired (static) strong law of large numbers. Although these concentration inequalities can yield sharper bounds, a downside is that extra conditions need to be put on an underlying density $f$. For example $f$ must have bounded support. For this reason, we have adopted a different approach via the Borel-Cantelli lemma, by which one can prove $n^{-1}\big(\chi_n(t) - \E[\chi_n(t)]\big) \to 0$ a.s.~by showing that the sum of the fourth moments is convergent. The relevant article taking an approach similar to ours is \cite{Goel2018}. 

The second contribution of this paper is to show the weak convergence of the process 
\[  
\xn(t) := n^{-1/2}\big( \chi_n(t)-\E[\chi_n(t)] \big), \ \ \ t \ge 0,
\]
with respect to the Skorohod $J_1$-topology. Proving finite-dimensional weak convergence of $\xn$ in conjunction with its tightness will allow us to obtain the desired convergence in $D[0, \infty)$. Finite-dimensional convergence will be established %for any function $\ell$ satisfying \ref{C1} and \ref{C2}
via the Cram{\'e}r-Wold device and Stein's method as in Theorem 2.4 in \cite{penr} by adhering closely to the proof of Theorem 3.9 in the same source. 
In addition, the tightness will be proven via Theorem 13.5 in \cite{Billingsley}.
These functional limit theorems enable us to capture dynamic features of topological changes in $D[0,\infty)$. The proofs for all results in this section are postponed to Section \ref{proofs}. 
%this represent the first unrestricted CLT for a topological invariant in the critical regime. All the better that the results hold for the computationally efficient Vietoris-Rips complex and that the weak convergence takes place in $D[0,\infty)$ so that dynamic features of topological changes can be captured. %%%%%%%This theorem needs to be modified. 

\bigskip
In order to obtain a clear picture of our limit theorems, it would be beneficial to start with some results on asymptotic moments of $\chi_n$. Define for $k_1, k_2 \in \Nat_0$, $t, s \ge 0$, and a Borel subset $A$ of $\R^d$,  
\begin{equation*} % \label{e:def.Psi}
\Psi_{k_1,k_2,A}(t,s) := \sum_{j=1}^{(k_1 \wedge k_2)+1} \psi_{j,k_1,k_2,A}(t,s), 
\end{equation*}
where $k_1 \wedge k_2 = \min\{ k_1, k_2 \}$, and 
\begin{align*}
\psi_{j,k_1,k_2,A}(t,s) &:= \frac{\int_A f(x)^{k_1+k_2+2-j}\dif x}{j! (k_1+1-j)! (k_2+1-j)!} \\ %\label{e:def.psi} 
&\quad \times 
\int_{(\R^d)^{k_1 + k_2+1 -j}} \hspace{-10pt} h_t^{k_1}(0,y_1,\dots,y_{k_1})h_s^{k_2}(0,y_1,\dots, y_{j-1}, y_{k_1+1}, \dots, y_{k_1+k_2+1-j})\dif \by. 
\end{align*}
Here we set $h_t^k(0,y_1,\dots,y_k)=1$ if $k=0$, so that $\Psi_{0,0,A}(t,s)=\psi_{1,0,0,A}(t,s)=\int_Af(x)dx$. In the sequel, we write $\Psi_{k_1,k_2}(t,s) := \Psi_{k_1,k_2,\R^d}(t,s)$ with $\psi_{j,k_1,k_2}(t,s) := \psi_{j,k_1,k_2,\R^d}(t,s)$. 
 %The proof of the proposition below is given in the next section. 
\begin{prop}  \label{p:cov.asym}
For $t,s \ge 0$, $A \subset \R^d$ open with $m(\partial A) = 0$, we have
\begin{align}
n^{-1} \E[\chi_{n,A}(t)] &\to \sum_{k=0}^\infty (-1)^k \psi_{k+1, k, k, A}(t,t), \ \ n\to\infty, \label{e:expectation.asym}\\
n^{-1} \Cov \big( \chi_{n,A}(t), \chi_{n,A}(s) \big) &\to \sum_{k_1=0}^\infty \sum_{k_2=0}^\infty (-1)^{k_1+k_2} \Psi_{k_1,k_2,A}(t,s), \ \ n\to\infty, \label{e:cov.asymp} 
\end{align}
so that both of the right hand sides are convergent for every such $A\subset \R^d$. 
\end{prop}

We can now introduce the FSLLN for the process $\chi_n$. 

\begin{thm}[FSLLN for $\chi_n$] \label{t:fslln}
As $n\to\infty$, 
$$
\frac{\chi_n(t)}{n} \to  \sum_{k=0}^\infty (-1)^k \psi_{k+1, k, k}(t,t) \ \ \text{a.s.~in } D[0,\infty), 
$$
where $D[0,\infty)$ is equipped with the uniform topology. 
\end{thm}

Before stating {our} functional central limit theorem (FCLT) for $\chi_n$, let us define its limiting process. First define $(\mathcal H_k, \, k \in \Nat_0)$ as a family of zero-mean Gaussian processes on a generic probability space $(\Omega, \mathcal{F}, \P)$, with intra-process covariance
\begin{equation}
\E[\mathcal{H}_k(t)\mathcal{H}_k(s)] = \Psi_{k, k}(t, s), \label{e:intraprocess}
\end{equation}
and inter-process convariance 
\begin{equation}
\E[\mathcal{H}_{k_1}(t)\mathcal{H}_{k_2}(s)] = \Psi_{k_1, k_2}(t, s), \label{e:interprocess}
\end{equation} 
for all $k, k_1, k_2 \in \Nat_0$ with $k_1 \neq k_2$ and $t, s \ge 0$. In the proof of Proposition \ref{p:cov.asym}, the functions $\Psi_{k_1,k_2}(t,s)$ naturally appear in the covariance calculation of $\chi_n$,  %; see \eqref{e:def.Psi} and \eqref{e:def.psi} for the formal definition of $\Psi_{k_1,k_2}(t,s)$. 
which in turn implies that the covariance functions in \eqref{e:intraprocess} and \eqref{e:interprocess} are well-defined. With these notations in mind, we now define the limiting Gaussian process for $\xn$ as 
\begin{equation}  \label{e:limit.GP}
\mathcal H(t) := \sum_{k=0}^\infty (-1)^k \mathcal H_k(t), \ \ t \ge 0, 
\end{equation}
such that
\begin{equation}  \label{e:limit.GP.cov}
\E[\mathcal H(t)\mathcal H(s)] = \sum_{k_1=0}^\infty \sum_{k_2=0}^\infty (-1)^{k_1+k_2} \Psi_{k_1,k_2}(t,s), \ \ t, s \ge0. 
\end{equation}
Once again, Proposition \ref{p:cov.asym} implies that the right hand side of \eqref{e:limit.GP.cov} can define the covariance functions of a limiting Gaussian process, since it is obtained as a (scaled) limit of the covariance functions of $\chi_n$. In particular, since \eqref{e:limit.GP.cov} is convergent, for every $t\ge 0$, $\mathcal H(t)$ is definable in the $L^2(\Omega)$-sense. Note that the Euler characteristic in \eqref{e:Euler.characteristic.proc} and the process \eqref{e:limit.GP} exhibit similar structure, in the sense that $S_k\big( \Pn, r_n(t) \big)$ in \eqref{e:Euler.characteristic.proc} and $\mathcal H_k(t)$ both correspond to the spatial distribution of $k$-simplices.
%\ed{via \eqref{e:intraprocess} and \eqref{e:interprocess}}. \ed{Notably, this implies that $\sum_{k_1=0}^\infty \sum_{k_2=0}^\infty (-1)^{k_1+k_2} \Psi_{k_1,k_2}(t,t) \geq 0$, which is the variance of $\mathcal{H}(t)$.} 

%\remove{
%Furthermore, we have the representation 
%\begin{equation} \label{e:lim_rep}
%\mathcal{H}^\ell(t) = \sum_{k=0}^{\infty} (-1)^k \mathcal{H}^\ell_k(t),
%\end{equation} 
%where $\big(\mathcal{H}^\ell_k(t), k \in \Nat_0\big)$ are mean-zero Gaussian processes with intra-process covariance function $\Psi^{\ell}_{k}(t_1, t_2) := \Psi^{\ell}_{k,k}(t_1, t_2)$, i.e.
%\[
%\E[\mathcal{H}^\ell_k(t_1)\mathcal{H}^\ell_k(t_2)] = \Psi^{\ell}_{k}(t_1, t_2),
%\]
%and inter-process covariance function
%\[
%\E[\mathcal{H}^\ell_{k_1}(t_1)\mathcal{H}^\ell_{k_2}(t_2)] = \Psi^{\ell}_{k_1, k_2}(t_1, t_2),
%\]
%for all $k, k_1, k_2 \in \Nat_0$ and $t_1, t_2 \geq 0$. This representation mirrors the limiting process the authors discovered for the critical regime in \cite{owathom}. %This is hardly surprising, as one can demonstrate that $\mathcal{H}^{\ell}_k$ is the limiting process of $\bar{S}^{\ell}_{k,n}(t) := n^{-1/2}(S^{\ell}_k(\Pn, r_n(t)) - \E[S^{\ell}_k(\Pn, r_n(t))])$. 
%}
Now, we proceed to stating the FCLT for $\chi_n$. 

\begin{thm}[FCLT for $\chi_n$] \label{t:main}
As $n \to \infty$, 
%Let $\ell: \mathbb{N}_0 \to \mathbb{N}_0$ be a non-decreasing function satisfying \ref{C1} and \ref{C2}. 
$$
\xn \Rightarrow \mathcal H \ \ \text{in } D[0,\infty),
$$
where $D[0,\infty)$ is equipped with the Skorohod $J_1$-topology. 
Furthermore, for every $0 < T < \infty$, we have that $\big(\mathcal{H}(t), \, 0 \leq t \leq T\big)$ has a continuous version with H{\"o}lder continuous sample paths of any exponent $\gamma \in [0, 1/2)$.
%$$
%\xn \overset{fidi}{\Rightarrow} \mathcal{H},
%$$
%meaning that for every $m \in \mathbb{N}$ and $0 < t_1 < t_2 < \dots < t_m < \infty$ we have %the following weak convergence in $\R^m$, 
%$$
%(\xn(t_i), \, i = 1, \dots, m) \Rightarrow (\mathcal{H}(t_i), \, i = 1, \dots, m).
%$$
%Moreover if an underlying simplicial complex is a Vietoris-Rips complex, this weak convergence takes place in $D[0,\infty)$. 
\end{thm}
\begin{rem}
The results of Theorem \ref{t:fslln} and Theorem \ref{t:main} also hold for the {\v C}ech complex, in the case of the latter theorem only up to finite-dimensional weak convergence of $\xn$. 
The definition of a $k$-simplex of the \v{C}ech complex requires a non-empty intersection of ``multiple'' closed balls. This makes it more difficult to establish the required tightness for the \v{C}ech complex. Specifically, obtaining bounds as in  Lemma \ref{l:hkl_cov} seems much harder. If one were able to establish such a nice bound, the rest of the argument for tightness would essentially be the same as the Vietoris-Rips case.
%The assumption of the Vietoris-Rips structure for our random simplex is needed for the proof of tightness. 
\remove{The primary reason for this is that it is much more difficult to establish bounds as in Lemma \ref{l:hkl_cov} in the {\v C}ech complex case. The reason for this may lie in the fact that it is not obvious how to find an upper bound of the form
\begin{align}
&\int_{(\R^d)^{k_1 + k_2 + 1 - j}} h_{s,t_1}^{k_1}(0, \mathbf{y}_0, \mathbf{y}_1) h_{t_2,s}^{k_2}(0, \mathbf{y}_0, \mathbf{y}_2) \dif{\mathbf{y}_0} \dif{\mathbf{y}_1} \dif{\mathbf{y}_2} \notag \\
&\phantom{\int_{(\R^d)^{k_1 + k_2 + 1 - j}} h_{s,t_1}^{k_1}(0, \mathbf{y}_0, \mathbf{y}_1) h_{t_2,s}^{k_2}}\leq C_{k_1, k_2, T, d} (F(t_2) - F(t_1))^2, \label{e:int_bound}
\end{align}
where $F$ is non-decreasing and continuous on $[0, T]$, as in Lemma \ref{l:hkl_cov}. It may be difficult to prove tightness in the {\v C}ech complex case using the same methods, i.e. decomposition of $h^k_{t,s}(0, \mathbf{y})$ into indicator functions, as in the Vietoris-Rips setup. See the proof section for more details. 
}

\end{rem}
\begin{ex}
Consider a map $x \mapsto \sup_{0\le t \le 1}|x(t)|$ from $D[0,1]$ to $\R_+$. This map is continuous on $C[0,1]$, the space of continuous functions on $[0,1]$. Since the limits in Theorems \ref{t:fslln}  and \ref{t:main} are both continuous, we get that as $n\to\infty$, 
\begin{align*}
&n^{-1} \sup_{0 \le t \le 1} |\chi_n(t)| \to \sup_{0 \le t \le 1}|\sum_{k=0}^\infty (-1)^k \psi_{k+1, k, k}(t,t)| \ \ a.s., \\
&n^{-1/2} \sup_{0 \le t \le 1} \big| \chi_n (t) -\E[\chi_n(t)] \big| \Rightarrow \sup_{0 \le t \le 1} |\mathcal H(t)|. 
\end{align*}
In particular, the latter claims that the supremum of a mean-centered Euler characteristic process can be approximated by $n^{1/2}\sup_{0 \le t \le 1}|\mathcal H(t)|$ for large enough $n$. 
\end{ex}
%\remove{As the finite-dimensional distributions of $\mathcal{H}^{\ell}$ have covariance matrices with entries as in \eqref{e:cov_spec}, Theorem~\ref{t:main} implies that if $\xt_n$ converges weakly in $D[0,\infty)$ the limit must be $\mathcal{H}^\ell$, as the mean and covariance functions completely specify a Gaussian process. In this section we prove that the distributions of $\xt^1_n$ are tight in $D[0, \infty)$ and hence $\xt_n \Rightarrow \mathcal{H}^1$ weakly in $D[0, \infty)$ as we see in the theorem below. That is, for Vietoris-Rips complexes we have convergence of the Euler characteristic process in Skorohod space. 
%
%
%\begin{thm} \label{t:main} 
%The sequence of distributions corresponding to the random functions $(\xt^1_n)_{n \geq 1}$ are tight and as a consequence, 
%\[
%\xt^1_n \Rightarrow \mathcal{H}^1, 
%\]
%weakly in $D[0, \infty)$.
%\end{thm}
%
%%From Theorem \ref{t:main} and the continuous mapping theorem, one can obtain weak convergence results such as $\sup_{t \in [0,T]} |\xt^1_n(t)| \Rightarrow \sup_{t \in [0,T]} |\mathcal{H}^1(t)|$. 
%
%\begin{rem} \label{r:cech}
%In light of Lemma \ref{l:hkl_prop} it would suffice to prove Lemma \ref{l:hkl_cov} for general $\ell$ to achieve convergence of $\xt^{\ell}_n$ for all $\ell$ in $D[0, \infty)$. In any case, it should be possible, especially given Theorem \ref{t:main} and H{\" o}lder-continous limiting process, to prove $\xt^{\ell}_n \Rightarrow \mathcal{H}^{\ell}$ weakly in $D[0,\infty)$.
%
%
%\end{rem}
%}
\section{Proofs} \label{proofs}

We first deal with moment asymptotics of $\chi_n$ in Section \ref{sec:moment.asym}. 
Section \ref{sec:proof.fslln} proves the FSLLN in Theorem \ref{t:fslln}. 
Subsequently, we establish Theorem \ref{t:main}, the proof of which is divided into two parts, with the first part devoted for finite-dimensional weak convergence, and the second for tightness. The proofs frequently refer to \emph{Palm theory for Poisson processes} for computing the moments of various Poisson functionals. A brief citation is given in Lemma \ref{l:palm} of the Appendix.
Finally we verify H\"older continuity of the limiting Gaussian process $\mathcal H$, following closely to what is established for subgraph counting processes in Proposition 4.2 of \cite{owadaFCLT}. 

For simplicity of description, we assume throughout this section that $ns_n^d=1$. However, generalizing it to $ns_n^d\to1$, $n\to\infty$ is straightforward. In the following, we write $a\vee b:=\max\{ a,b\}$ and $a \wedge b := \min\{ a,b \}$ for $a, b \in \R$. 

\subsection{Proof of moment asymptotics} \label{sec:moment.asym}

%%%%%%%%%%%%%%%%%%%%%%%%%%%%%%%%%%%%%%%%%%%%%%%%%%%%%%%%%%%%%%%
%%%%%%%%%%%%%%%%%%%%%%%%%%%%%%%%%%%%%%%%%%%%%%%%%%%%%%%%%%%%%%%
%%%%%%%%%%%%%%%%%%%%%%%%%%%%%%%%%%%%%%%%%%%%%%%%%%%%%%%%%%%%%%%
%%%%%%%%%%%%%%%%   E L E M E N T A R Y    C O V A R I A N C E    R E S U L T   %%%%%%%%%%%%%%%%
%%%%%%%%%%%%%%%%%%%%%%%%%%%%%%%%%%%%%%%%%%%%%%%%%%%%%%%%%%%%%%%
%%%%%%%%%%%%%%%%%%%%%%%%%%%%%%%%%%%%%%%%%%%%%%%%%%%%%%%%%%%%%%%
%%%%%%%%%%%%%%%%%%%%%%%%%%%%%%%%%%%%%%%%%%%%%%%%%%%%%%%%%%%%%%%

Without loss of generality, the proof of Proposition \ref{p:cov.asym} only handles the case when $A=\R^d$. Throughout this section, let $\Y$, $\Y_1$, and $\Y_2$ denote collections of iid random points with density $f$. We begin with the following lemma.  
\begin{lem}  \label{l:indicator.asym}
$(i)$ For $t\ge 0$ we have, as $n\to\infty$, 
$$
\frac{n^k}{(k+1)!}\, \E\ \big[ h_{r_n(t)}^k (\Y) \big] \to \psi_{k+1,k,k}(t,t). 
$$
$(ii)$ For all $n \in \Nat$, 
$$
n^k \E\ \big[ h_{r_n(t)}^k (\Y) \big] \le (a_t)^k,
$$
where 
\begin{equation} \label{e:def.at}
a_t := (2t)^d \theta_d \|f \|_\infty
\end{equation} 
with $\theta_d=m\big( B(0,1) \big)$, i.e., volume of the unit ball in $\R^d$. \\
$(iii)$ For $1\le j \le (k_1\wedge k_2)+1$, $k_1, k_2 \in \Nat_0$, and $t,s\ge 0$, 
$$
\frac{n^{k_1+k_2+1-j}}{j! (k_1+1-j)! (k_2+1-j)!}\, \E \big[ h_{r_n(t)}^{k_1}(\Y_1)h_{r_n(s)}^{k_2}(\Y_2)\, \one \big\{ |\Y_1 \cap \Y_2|=j \big\} \big] \to \psi_{j,k_1,k_2}(t,s)
$$
as $n\to\infty$. \\
$(iv)$ For all $n\in \Nat$, 
$$
n^{k_1+k_2+1-j} \E \big[ h_{r_n(t)}^{k_1}(\Y_1)h_{r_n(s)}^{k_2}(\Y_2)\, \one \big\{ |\Y_1 \cap \Y_2|=j \big\} \big] \le (a_{t\vee s})^{k_1+k_2+1-j}. 
$$
\end{lem}
\begin{proof}
We shall prove $(iii)$ and $(iv)$ only, since $(i)$ and $(ii)$ can be established by a similar and simpler argument. 
Making change of variables $x_1=x$ and $x_i = x+s_n y_{i-1}$, $i=1,\dots,k_1+k_2+2-j$, the left hand side of $(iii)$ equals
\begin{align}
&\frac{n^{k_1+k_2+1-j}}{j! (k_1+1-j)! (k_2+1-j)!}\, \int_{(\R^d)^{k_1+k_2+2-j}} h_{r_n(t)}^{k_1}(x_1,\dots,x_{k_1+1}) \notag \\
&\qquad \qquad \qquad\qquad \qquad\qquad   \times h_{r_n(s)}^{k_2}(x_1,\dots,x_j,x_{k_1+2}, \dots, x_{k_1+k_2+2-j}) \prod_{i=1}^{k_1+k_2+2-j} f(x_i)\dif \bx \notag \\
&=\frac{(ns_n^d)^{k_1+k_2+1-j}}{j! (k_1+1-j)! (k_2+1-j)!}\, \int_{\R^d} \int_{(\R^d)^{k_1+k_2+1-j}} h_t^{k_1}(0,y_1\dots,y_{k_1}) \label{e:long.lemma}\\
&\quad \times h_s^{k_2}(0,y_1,\dots,y_{j-1},y_{k_1+1}, \dots, y_{k_1+k_2+1-j}) f(x)\prod_{i=1}^{k_1+k_2+1-j} f(x+s_ny_i) \dif \by \dif x. \notag 
\end{align}
Recall that $ns_n^d=1$ and note that $\prod_{i=1}^{k_1+k_2+1-j}f(x+s_ny_i) \to f(x)^{\kokt +1-j}$, $n\to\infty$, holds under the integral sign because of the Lebesgue differentiation theorem. Thus, \eqref{e:long.lemma} converges to $\psi_{j,k_1,k_2}(t,s)$ as $n\to\infty$.  

Now let us turn to proving statement $(iv)$. Without loss of generality, we may assume $s\le t$. Performing the same change of variables as in $(iii)$, the left hand side of $(iv)$ is bounded by 
\begin{equation}  \label{e:expression.upper.bound}
\big(  \|f\|_\infty \big)^{\kokt+1-j}\int_{(\R^d)^{\kokt+1-j}} h_t^{k_1}(0,y_1\dots,y_{k_1}) 
 h_s^{k_2}(0,y_1,\dots,y_{j-1},y_{k_1+1}, \dots, y_{k_1+k_2+1-j}) \dif \by. 
\end{equation}
By the definition of the indicators $h_t^{k_1}$, $h_s^{k_2}$, each of the $y_i$'s in \eqref{e:expression.upper.bound} must be distance at most $2t$ from the origin. Therefore, \eqref{e:expression.upper.bound} can be bounded by 
$$
\big(  \|f\|_\infty \big)^{\kokt +1-j} m\big( B(0,2t) \big)^{\kokt+1-j} = (a_t)^{\kokt+1-j}. 
$$
\end{proof}
\begin{proof}[Proof of Proposition \ref{p:cov.asym}]
We only prove \eqref{e:cov.asymp} as the proof techniques for \eqref{e:expectation.asym} are very similar to \eqref{e:cov.asymp}. Specifically, we shall make use of $(ii)$, $(iii)$, and $(iv)$ of Lemma \ref{l:indicator.asym}. %whereas one instead needs $(i)$ and $(ii)$ for \eqref{e:expectation.asym}. 
We start by writing
\begin{align}
n^{-1} \Cov \big( \chi_n(t), \chi_n(s) \big) &= n^{-1} \E \bigg[ \sum_{k_1=0}^\infty\sum_{k_2=0}^\infty (-1)^{\kokt} S_{k_1}\big(\Pn, r_n(t)\big) S_{k_2}\big(\Pn, r_n(s)\big) \bigg] \label{e:comp.covariance}\\
&\qquad- n^{-1} \E\bigg[ \sum_{k=0}^\infty(-1)^k S_k\big(\Pn, r_n(t)\big) \bigg]\E\bigg[ \sum_{k=0}^\infty(-1)^k S_k\big(\Pn, r_n(s)\big) \bigg].\notag
\end{align}
Next, Palm theory for Poisson processes, Lemma \ref{l:palm} $(ii)$, along with the bounds given in Lemma \ref{l:indicator.asym} $(ii)$ and $(iv)$, yields that 
\begin{align*}
&\E \Big[  S_{k_1}\big(\Pn, r_n(t)\big) S_{k_2}\big(\Pn, r_n(s)\big) \Big] \\
&=\sum_{j=0}^{(k_1 \wedge k_2)+1} \E\bigg[ \sum_{\Y_1\subset \Pn}\sum_{\Y_2\subset \Pn} h_{r_n(t)}^{k_1}(\Y_1)h_{r_n(s)}^{k_2}(\Y_2)\, \one\big\{ |\Y_1\cap \Y_2|=j \big\} \bigg] \\
&= \frac{n^{\kokt +2}}{(k_1+1)!(k_2+1)!}\, \E \big[ h_{r_n(t)}^{k_1}(\Y_1) \big]\E \big[ h_{r_n(s)}^{k_2}(\Y_2) \big] \\
&\qquad  + \sum_{j=1}^{(k_1 \wedge k_2)+1}\frac{n^{\kokt+2-j}}{j!(k_1+1-j)!(k_2+1-j)!}\, \E\Big[ h_{r_n(t)}^{k_1}(\Y_1) h_{r_n(s)}^{k_2}(\Y_2) \ind{\card{\Y_1 \cap \Y_2} = j}  \Big] \\
&\le \frac{n^2(a_t)^{k_1}(a_s)^{k_2}}{(k_1+1)!(k_2+1)!} +  \sum_{j=1}^{(k_1 \wedge k_2)+1}\frac{n(a_{t\vee s})^{\kokt+1-j}}{j!(k_1+1-j)!(k_2+1-j)!}. 
\end{align*}
Here it is straightforward to see that 
\begin{align*}
&\sum_{k=0}^\infty \frac{(a_t)^k}{(k+1)!} < e^{a_t} <\infty, \ \ \sum_{k_1=0}^\infty \sum_{k_2=0}^\infty  \sum_{j=1}^{(k_1 \wedge k_2)+1}\frac{(a_{t\vee s})^{\kokt+1-j}}{j!(k_1+1-j)!(k_2+1-j)!} < 2 e^{3a_{t\vee s}} < \infty.  
\end{align*}
% (3/9/2020) NOTE: I COULD ALWAYS BE WRONG HERE, BUT I COULDN'T FIND A WAY TO BOUND THE RIGHTMOST SUM BY 2e^{3a} AS YOU HAD GIVEN, THIS WAS THE BEST I COULD DO
So Fubini's theorem is applicable to the first term in \eqref{e:comp.covariance}. Repeating the same argument for the second term of \eqref{e:comp.covariance}, one can get 
\begin{align*}
n^{-1} \Cov \big( \chi_n(t), \chi_n(s) \big) &= \sum_{k_1=0}^\infty \sum_{k_2=0}^\infty(-1)^{\kokt} \sum_{j=1}^{(k_1 \wedge k_2)+1}\frac{n^{\kokt+1-j}}{j!(k_1+1-j)!(k_2+1-j)!}\, \\
&\qquad \qquad \qquad \qquad \times \E\Big[ h_{r_n(t)}^{k_1}(\Y_1) h_{r_n(s)}^{k_2}(\Y_2) \ind{\card{\Y_1 \cap \Y_2} = j}  \Big].
\end{align*}
By virtue of Lemma \ref{l:indicator.asym} $(iii)$ and $(iv)$, the dominated convergence theorem can conclude that the last expression converges to $\sum_{k_1=0}^\infty \sum_{k_2=0}^\infty(-1)^{\kokt} \Psi_{k_1,k_2}(t,s)$ as required. 
\end{proof}
\subsection{Proof of FSLLN}  \label{sec:proof.fslln}

To prove the functional strong law of large numbers, we first establish a result which allows us to extend a ``pointwise'' strong law  for a fixed $t$ into a functional one, if the processes are non-decreasing and there is a deterministic and continuous limit. 
We again would like to emphasize that our approach in this section gives an improvement from the viewpoints of assumptions on the density $f$. Unlike the existing results such as \cite{Yogeshwaran2017}, we do not require $f$ to have compact and convex support.
\begin{prop}\label{p:monoFSLLN}
Let $(X_n, \, n \in \mathbb{N})$ be a sequence of random elements in $D[0, \infty)$ with non-decreasing sample paths. Suppose $\lambda: [0,\infty) \to \R$ is a continuous and non-decreasing function. If we have 
\begin{equation} \label{e:slln}
X_n(t) \to \lambda(t),  \ \ n\to\infty, \ \ \text{a.s.}
\end{equation}
for every $t\ge0$, then it follows that 
\begin{equation*} %\label{e:fslln}
\sup_{t \in [0,T]} |X_n(t) - \lambda(t)| \to 0,  \ \ n\to\infty, \ \ \text{a.s.}
\end{equation*}
for every $0\le T<\infty$. Hence, it holds that $X_n \to \lambda$ a.s.~in $D[0, \infty)$ endowed with the uniform topology.  
\end{prop}
\begin{proof}
Fix $0\le T<\infty$. Note that $\lambda$ is uniformly continuous on $[0,T]$. Given $\epsilon>0$, choose $k=k(\epsilon)\in \mathbb{N}$ such that  for all $s,t\in [0,T]$, 
\begin{equation}  \label{e:unif.conti}
|s-t|\le 1/k \text{ implies }  \big|\lambda(s)-\lambda(t)  \big| < \epsilon. 
\end{equation}
Since $X_n(t)$ and $\lambda(t)$ are both non-decreasing in $t$, we see that 
\begin{align*}
&\sup_{t\in [0,T]} \big| X_n(t)-\lambda(t) \big| =\max_{1\le i \le k} \sup_{t\in [ (i-1)T/k, \, iT/k ]} \big| X_n(t)-\lambda(t) \big| \\
&\qquad \le \max_{1\le i \le k} \bigg\{ \Big( X_n(iT/k)-\lambda((i-1)T/k) \Big) \vee \Big( \lambda(iT/k)-X_n((i-1)T/k) \Big) \bigg\}  \\
&\qquad \le \max_{1\le i \le k} \bigg\{ \Big( X_n(iT/k)-\lambda(iT/k) \Big) \vee \Big( \lambda((i-1)T/k)-X_n((i-1)T/k) \Big) \bigg\} + \epsilon \\
&\qquad \le \max_{0\le i \le k} \Big| X_n(iT/k) -\lambda(iT/k) \Big| + \epsilon, 
\end{align*}
where the second inequality follows from \eqref{e:unif.conti}. By the SLLN in \eqref{e:slln}, the last expression tends to $\epsilon$ almost surely as $n\to\infty$. Since $\epsilon$ is arbitrary, we can complete the proof. 
\end{proof}
%\remove{
%\begin{proof}
%We begin by noting that since $\lambda$ is continous on $[0,T]$ it is also uniformly continuous. 
%Now, we note that for any $k \in \mathbb{N}$, 
%\begin{align}
%&\sup_{t \in [0,T]} |X_n(t) - \lambda(t)| \notag \\
%&\qquad= \max_{1 \leq i \leq k} \, \sup_{t \in \big[\frac{(i-1)T}{k}, \, \frac{iT}{k}\big]} |X_n(t) - \lambda(t)| \notag \\
%&\qquad\leq \max_{1 \leq i \leq k} \Big(X_n(iT/k) - \lambda((i-1)T/k) \vee (\lambda(iT/k) - X_n((i-1)T/k)) \Big). \label{e:tkbound}
%\end{align}
%If we choose $k = k(\epsilon)$ s.t. $|s-t| \leq 1/k$ implies that $|\lambda(s) - \lambda(t)| < \epsilon$ (which we can do by $\lambda$ uniformly continuous), then
%\[
%X_n(iT/k) - \lambda((i-1)T/k) \leq X_n(iT/k) - \lambda(iT/k) + \epsilon,
%\]
%and
%\[
%\lambda(iT/k) - X_n((i-1)T/k)) \leq \lambda((i-1)T/k) - X_n((i-1)T/k)) + \epsilon,
%\]
%so that \eqref{e:tkbound} can be bounded by 
%\begin{align*}
%&\max_{1 \leq i \leq k} \Big((X_n(iT/k) - \lambda(iT/k)) \vee \lambda((i-1)T/k) - X_n((i-1)T/k)\Big) + 2\epsilon \\
%&\qquad \leq \max_{0 \leq i \leq k} |X_n(iT/k) - \lambda(iT/k)| + 2\epsilon,
%\end{align*}
%by the commutativity of the maximum. If we have that
%\begin{equation}\label{e:ktrunc}
%\lim_{n \to \infty} \max_{0 \leq i \leq k} |X_n(iT/k) - \lambda(iT/k)| = 0, \quad k \in \mathbb{N}, 
%\end{equation}
%then by the above
%\[
%\limsup_{n\to \infty} \sup_{t \in [0,T]} |X_n(t) - \lambda(t)| \leq 2\epsilon, 
%\]
%which implies \eqref{e:fslln}, because the strong law for each $t \geq 0$ implies \eqref{e:ktrunc} has probability 1. 
%\end{proof}
%}

\begin{proof}[Proof of Theorem \ref{t:fslln}]
Since \eqref{e:Euler.characteristic.proc} is almost surely represented as a sum of finitely many terms, it can be split into two parts, 
$$
\chi_n(t) = \sum_{k=0}^\infty S_{2k}\big( \Pn, r_n(t) \big) - \sum_{k=0}^\infty S_{2k+1}\big( \Pn, r_n(t) \big) =: \chi_n^{(1)}(t) - \chi_n^{(2)}(t) \ \ \text{a.s.}
$$
Denoting by $K(t)$ the limit of \eqref{e:expectation.asym} with $A=\R^d$, we decompose it in a way similar to the above,
$$
K(t)=\sum_{k=0}^\infty \psi_{2k+1, 2k, 2k}(t,t) - \sum_{k=0}^\infty \psi_{2k+2, 2k+1, 2k+1}(t,t) =: K^{(1)}(t)-K^{(2)}(t). 
$$
Our final goal is to prove that for every $0<T<\infty$, 
$$
\sup_{0\le t \le T}\Big| \frac{\chi_n(t)}{n} - K(t) \Big|\to 0, \ \ n\to\infty, \ \ \text{a.s.}, 
$$
which is clearly implied by 
\begin{equation*} %\label{e:slln12}
\sup_{0\le t \le T}\Big| \frac{\chi_n^{(i)}(t)}{n} - K^{(i)}(t) \Big|\to 0, \ \ n\to\infty, \ \ \text{a.s.}
\end{equation*}
for each $i=1,2$. As $\chi_n^{(i)}(t)/n$ and $K^{(i)}(t)$ satisfy the conditions of Proposition \ref{p:monoFSLLN}, it suffices to show that 
\[
\frac{\chi_n^{(i)}(t)}{n} \to K^{(i)}(t), \ \ n\to\infty, \ \ \text{a.s.},
\]
for every $t \geq 0$. We will only prove the case $i=1$, and henceforth omit the superscript $(1)$ from $\chi_n^{(1)}(t)$ and $K^{(1)}(t)$. 
It then suffices to show that 
\begin{equation}
n^{-1} | \chi_n(t) - \E [\chi_n(t)] | \to 0, \ \ n\to\infty, \ \ \text{a.s.}, \label{e:1st.sup} 
\end{equation}
and
\begin{equation}
\big| n^{-1}\E[\chi_n(t)] - K(t) \big| \to 0, \ \ n\to\infty. \label{e:3rd.sup}
\end{equation}
First we will deal with \eqref{e:3rd.sup}. It follows from the customary change of variables as in the proof of Lemma \ref{l:indicator.asym}, that 
\begin{align*}
& \big| n^{-1}\E[\chi_n(t)] - K(t) \big| \\
&\qquad= \bigg| \sum_{k=1}^\infty \frac{1}{(2k+1)!}\, \int_{\R^d} \int_{(\R^d)^{2k}} h_{t}^{2k}(0,y_1,\dots,y_{2k})  \\
&\qquad \qquad \qquad \qquad \qquad \qquad \times f(x)\Big( \prod_{i=1}^{2k}f(x+s_ny_i)-f(x)^{2k} \Big) \dif \by \dif x \bigg| \\
&\qquad\le \sum_{k=1}^\infty \frac{1}{(2k+1)!}\, \int_{\R^d} \int_{(\R^d)^{2k}} h_{t}^{2k}(0,y_1,\dots,y_{2k}) f(x) \Big| \prod_{i=1}^{2k}f(x+s_ny_i)-f(x)^{2k} \Big| \dif \by \dif x. 
\end{align*}
Similarly to the proof of Lemma \ref{l:indicator.asym} $(ii)$ or $(iv)$, one can show that the last term above is bounded by $2\sum_{k=1}^\infty (a_t)^{2k}/(2k+1)! < \infty$ ($a_t$ is defined in \eqref{e:def.at}). Thus, the dominated convergence theorem concludes \eqref{e:3rd.sup}.

Now, let us turn our attention to \eqref{e:1st.sup}. From the Borel-Cantelli lemma it suffices to show that, for every $\epsilon>0$, 
$$
\sum_{n=1}^\infty \P \Big( \big| \chi_n(t)-\E[\chi_n(t)] \big| >\epsilon n\Big) <\infty. 
$$
By Markov's inequality, the left hand side above is bounded by 
$$
\frac{1}{\epsilon^4}\sum_{n=1}^\infty \frac{1}{n^4}\E \Big[ \big( \chi_n(t)-\E[\chi_n(t)] \big)^4 \Big]. 
$$
Since $\sum_n n^{-2}<\infty$, we only need to show that 
\begin{equation}  \label{e:finite.4th.moment}
\limsup_{n\to\infty}\frac{1}{n^2} \E \Big[ \big( \chi_n(t)-\E[\chi_n(t)] \big)^4 \Big] <\infty. 
\end{equation}
Applying Fubini's theorem as in the proof of Proposition \ref{p:cov.asym}, along with H\"older's inequality, we get that
\begin{align*}
&\frac{1}{n^2}\, \E \Big[ \big( \chi_n(t)-\E[\chi_n(t)] \big)^4 \Big] \\
&=\frac{1}{n^2} \sum_{(k_1,\dots,k_4)\in \Nat^4} \E \bigg[ \prod_{i=1}^4 \Big( S_{2k_i}\big( \Pn, r_n(t) \big) -E\big[S_{2k_i} \big( \Pn, r_n(t) \big)  \big]\Big)  \bigg] \\
&\le  \bigg[ \sum_{k=1}^\infty \bigg\{\frac{1}{n^2}  \E \Big[ \Big( S_{2k}\big( \Pn, r_n(t) \big) - \E \big[ S_{2k}\big( \Pn, r_n(t) \big) \big] \Big)^4  \Big] \bigg\}^{1/4}  \bigg]^4. 
\end{align*}
Now, \eqref{e:finite.4th.moment} can be obtained if we show that 
\begin{equation}  \label{e:sum.limsup.4th.moment}
\sum_{k=1}^\infty  \bigg\{ \limsup_{n\to\infty}\frac{1}{n^2}\,  \E \Big[ \Big( S_{2k}\big( \Pn, r_n(t) \big) - \E \big[ S_{2k}\big( \Pn, r_n(t) \big) \big] \Big)^4  \Big] \bigg\}^{1/4} <\infty. 
\end{equation}

From this point on, let us introduce a shorthand notation, $S_{2k}:=S_{2k}\big( \Pn, r_n(t) \big)$. In order to find an appropriate upper bound for \eqref{e:sum.limsup.4th.moment}, by the binomial expansion we write 
\begin{equation}  \label{e:binom.exp.4th.moment}
\E \big[ \big(S_{2k}-\E[S_{2k}]\big)^4 \big] = \sum_{\ell=0}^4 \binom{4}{\ell}(-1)^\ell \E[S_{2k}^\ell] \big( \E[S_{2k}] \big)^{4-\ell}. 
\end{equation}
For every $\ell\in \{ 0,\dots,4 \}$, one can denote $\E[S_{2k}^\ell] \big( \E[S_{2k}] \big)^{4-\ell}$ as
\begin{equation}  \label{e:general.4th.moment}
\E \bigg[ \sum_{\Y_1 \subset \Pn^{(1)}}\sum_{\Y_2 \subset \Pn^{(2)}}\sum_{\Y_3 \subset \Pn^{(3)}}\sum_{\Y_4 \subset \Pn^{(4)}} \prod_{i=1}^4 h_{r_n(t)}^{2k}(\Y_i) \bigg], 
\end{equation}
where for every $i, j \in \{ 1,\dots,4 \}$, we have either $\Pn^{(i)}=\Pn^{(j)}$ or $\Pn^{(i)}$ is an independent copy of $\Pn^{(j)}$. 
If $|\Y_1 \cup \cdots \cup \Y_4|=8k+4$, i.e., $\Y_1, \dots, \Y_4$ do not have any common elements, Palm theory (Lemma \ref{l:palm}) shows that \eqref{e:general.4th.moment} is equal to $\big( \E[S_{2k}] \big)^4$, which grows at the rate of $O(n^4)$ (see Lemma \ref{l:indicator.asym} $(i)$). In this case, the total contribution to \eqref{e:binom.exp.4th.moment} disappears, because 
$$
\sum_{\ell=0}^4 \binom{4}{\ell}(-1)^\ell \big( \E[S_{2k}] \big)^4=0. 
$$
Suppose next that $|\Y_1 \cup \cdots \cup \Y_4|=8k+3$, that is, there is exactly one common element between $\Y_i$ and $\Y_j$ for some $i\neq j$ with no other overlappings. Then \eqref{e:general.4th.moment} is equal to
$$
\E \bigg[ \sum_{\Y_1\subset \Pn}\sum_{\Y_2\subset \Pn} h_{r_n(t)}^{2k}(\Y_1)h_{r_n(t)}^{2k}(\Y_2)\, \one \{ |\Y_1 \cap \Y_2|=1 \} \bigg] \big( \E[S_{2k}] \big)^2. 
$$
Although the growth rate of the above term is $O(n^3)$ (see Lemma \ref{l:indicator.asym} $(i)$ and $(iii)$), an overall contribution to \eqref{e:binom.exp.4th.moment} is again canceled. This is  because  
\begin{align*}
&\Big\{ \binom{4}{2}(-1)^2+\binom{4}{3}(-1)^3\binom{3}{2}+\binom{4}{4}(-1)^4\binom{4}{2} \Big\} \\
&\quad \times \E \bigg[ \sum_{\Y_1\subset \Pn}\sum_{\Y_2\subset \Pn} h_{r_n(t)}^{2k}(\Y_1)h_{r_n(t)}^{2k}(\Y_2)\, \one \{ |\Y_1 \cap \Y_2|=1 \} \bigg] \big( \E[S_{2k}] \big)^2=0. 
\end{align*}

By the above discussion, we only need to consider the case where there are at least two common elements within $\Y_1, \dots, \Y_4$. Among many such cases, let us deal with a specific term, 
\begin{align}
&n^{-2} \E \bigg[ \sum_{\Y_1\subset \Pn}\sum_{\Y_2\subset \Pn}\sum_{\Y_3\subset \Pn}\sum_{\Y_4\subset \Pn} \prod_{i=1}^4 h_{r_n(t)}^{2k}(\Y_i)\,  \label{e:even.split} \\
&\qquad \qquad \qquad \times \one \big\{ |\Y_1\cap \Y_2|=\ell_1, \,|\Y_3\cap \Y_4|=\ell_2, \,  \big| (\Y_1\cup \Y_2)\cap (\Y_3\cup \Y_4) \big|=0 \big\} \bigg], \notag 
\end{align}
where $\ell_1, \ell_2 \in \{ 1,\dots, 2k+1 \}$. Palm theory allows us to write \eqref{e:even.split} as
\begin{equation}  \label{e:after.Palm}
\prod_{i=1}^2 \frac{n^{4k+1-\ell_i}}{\ell_i! \big( (2k+1-\ell_i)! \big)^2}\, \E\Big[ h_{r_n(t)}^{2k}(\Y_1)h_{r_n(t)}^{2k}(\Y_2)\, \one \{ |\Y_1\cap \Y_2|=\ell_i \} \Big] . 
\end{equation}
By Lemma \ref{l:indicator.asym} $(iv)$ and $\ell! (2k+1-\ell)! \geq k!$ for any $\ell \in \{ 1,\dots,2k+1 \}$, one can bound \eqref{e:after.Palm} by
$$
\prod_{i=1}^2 \frac{(a_t)^{4k+1-\ell_i}}{\ell_i ! \big( (2k+1-\ell_i)! \big)^2} \le \frac{(a_t)^{8k+2-\ell_1-\ell_2}}{k!}. 
$$
Now, the ratio test shows that 
$$
\sum_{k=1}^\infty \bigg\{ \frac{(a_t)^{8k+2-\ell_1-\ell_2}}{k!} \bigg\}^{1/4} < \infty
$$
as desired. 
Notice that all the cases except \eqref{e:even.split} can be handled in a very similar way, and so, \eqref{e:sum.limsup.4th.moment} follows. 
\end{proof}

\subsection{Proof of finite-dimensional convergence in Theorem \ref{t:main}} \label{sec:fidi.dim.proof}
\begin{proof}[Proof of finite-dimensional convergence in Theorem \ref{t:main}]
Throughout the proof, $C^*$ denotes a generic positive constant that potentially varies across and within the lines. 
 Recall \eqref{e:left.most.Euler.characteristic} and define $\xnA(t)$ analogously to $\xn(t)$ by mean-centering and scaling by $n^{-1/2}$. We first consider the case where $A$ is an open and bounded subset of $\R^d$ with $m(\partial A) = 0$. 

%and \ed{write $\chi_n(t) := \chi_{n,\R^d}(t)$.}
%\remove{$$
%\chi_{n,A}^M(t) := \sum_{k=0}^M(-1)^k \sum_{\Y\subset \Pn} h_{r_n(t)}^k \one \big\{ \text{LMP}(\Y)\in A \big\}, \ \ t \ge 0, \ M\in \Nat_0. 
%$$}
%\st{Define also $\xnAM(t)$ and $\xnM(t)$ analogously to $\xn(t)$ by mean-centering } 

From the viewpoint of the Cram{\'e}r-Wold device, one needs to establish weak convergence of $\sum_{i=1}^m a_i \xn(t_i)$ for every $0 <t_1 < \cdots <t_m$, $m \in \Nat$, and $a_i\in \R$, $i=1,\dots,m$. Our proof exploits Stein's normal approximation method in Theorem 2.4 of \cite{penr}.
Let $(Q_{j,n}, \, j \ge 1)$ be an enumeration of disjoint subsets of $\R^d$ congruent to $(0,r_n(t_m)]^d$, such that $\R^d = \bigcup_{j=1}^\infty Q_{j,n}$. 
Let $H_n = \{j \in \mathbb{N}: Q_{j,n} \cap A \neq \emptyset\}$. Define 
$$
\xi_{j,n} := \sum_{k=0}^{\infty} (-1)^k \sum_{\Y \subset \Pn} \sum_{i=1}^m a_i  h_{r_n(t_i)}^k(\Y)\one \big\{\text{LMP}(\Y) \in A \cap Q_{j,n}\big\}, 
$$
and also, 
$$
\barxi_{j,n} := \frac{\xi_{j,n} - \E[\xi_{j,n}]}{\sqrt{\Var\big(\sum_{i=1}^m a_i \chi_{n,A}(t_i)\big)}}.
$$
Then, we have $\sum_{i=1}^m a_i \chi_{n,A}(t_i) = \sum_{j \in H_n} \xi_{j,n}$. 

Now, let us define $H_n$ to be the vertex set of a \emph{dependency graph} (see Section 2.1 of \cite{penr} for the formal definition) for the random variables 
$(\barxi_{j,n}, \, j \in H_n)$ by setting $j \sim \pr{j}$ if and only if the condition
\[
\inf\big\{ \norm{x-y} : x \in Q_{j,n}, \, y \in Q_{\pr j, n}\big\} \leq 4r_n(t_m),
\]
is satisfied. This is because $\xi_{j,n}$ and $\xi_{j', n}$ become independent whenever $j\sim j'$ fails to hold. 
%$h^{k,\ell}_{r_n(t_i)}(\Y)\ind{ l(\Y) \in Q_{j,n} \cap A}$ only depends on the set 
%\[
%\text{Tube}(Q_{j,n}, 2r_n(t_m)) := \big \{ x \in \R^d: \inf_{y \in Q_{j,n}} \norm{x-y} \leq 2r_n(t_m)  \bigr\}. 
%\]
Now we must ensure that the other conditions of Theorem~2.4 in \cite{penr} are satisfied with respect to the dependency graph $(H_n, \sim)$. First, $\sum_{j\in H_n}\barxi_{j,n}$ is a zero-mean random variable with unit variance.  We know that $|H_n| = O(s_n^{-d})$ as $A$ is bounded. Furthermore, the maximum degree of any vertex of $H_n$ is uniformly bounded by a positive and finite constant. Let $Z$ denote a standard normal random variable. Then the aforementioned theorem implies that 
\begin{align}
\Bigl|\,  &\P\bigl( \sum_{j \in H_n} \barxi_{j,n} \leq x \bigr) - \P(Z \leq x) \, \Bigr| \leq C^*\left( \sqrt{s_n^{-d} \max_j \, \E \bigl[ |\barxi_{j,n}|^3 \bigr]} + \sqrt{s_n^{-d} \max_j\, \E \bigl[ |\barxi_{j,n}|^4 \bigr]}   \right) \notag\\
&\leq C^*\left( \sqrt{s_n^{-d}n^{-3/2} \max_j \, \E \bigl[ |\xi_{j,n}-\E[\xi_{j,n}]|^3 \bigr]} + \sqrt{s_n^{-d} n^{-2} \max_j\, \E \bigl[ |\xi_{j,n} -\E[ \xi_{j,n}]|^4 \bigr]}   \right), \label{e:stein_bound}
\end{align}
where the second inequality follows from Proposition \ref{p:cov.asym} that claims that $\Var \big(  \sum_{i=1}^m a_i \chi_{n,A}(t_i)) \big)$ is asymptotically equal to $n$ up to multiplicative constants. 
Minkowski's inequality implies that 
\begin{align*}
\Big (\E \bigl[ |\xi_{j,n}-\E[\xi_{j,n}]|^p \bigr]\Big)^{1/p} \leq \big(\E\big[|\xi_{j,n}|^p\big]\big)^{1/p} + \E\big[|\xi_{j,n}|\big]. 
\end{align*}
Recall that for fixed $\Y \subset \R^d$, $h^{k}_t(\Y)$ is non-decreasing in $t$. 
Then, we have that 
\begin{align*}
|\xi_{j,n}| &\leq \sum_{k=0}^{\infty} \sum_{\Y \subset \Pn} \sum_{i=1}^m |a_i | h^{k}_{r_n(t_i)}(\Y)\ind{\text{LMP}(\Y) \in A \cap Q_{j,n}} \\
&\leq C^* \sum_{k=0}^{\infty} \sum_{\Y \subset \Pn} h^k_{r_n(t_m)}(\Y) \ind{\text{LMP}(\Y) \in A \cap Q_{j,n}}   \\
&\leq C^* \sum_{k=0}^{\infty} \dbinom{\Pn\big(\text{Tube}(Q_{j,n}, 2r_n(t_m))\big)}{k+1} \\
&\leq C^* \cdot 2^{\Pn(\text{Tube}(Q_{j,n}, \, 2r_n(t_m)))}, 
\end{align*}
where 
$$
\text{Tube}\big(Q_{j,n}, 2r_n(t_m)\big) = \big \{ x \in \R^d: \inf_{y \in Q_{j,n}} \norm{x-y} \leq 2r_n(t_m)  \bigr\}. 
$$
%where $x_{j,n}$ is the ``center'' of $Q_{j,n}$. That is, if $Q_{j,n} = z + (0, r_n(t_m)]^d$ for $z \in r_n(t_m)\mathbb{Z}^d$, then $x_{j,n} = z + r_n(t_m)/2$. To explain a little bit more of the above inequality, know that if $\Y$ forms a $k$-simplex in $\mathcal{K}^\ell(\Y, r_n(t_m))$ and the left-most point of $\Y$ is in $Q_{j,n}$ then all the points of $Y$ must be within $2r_n(t_m)$ of the left-most point and hence within $2r_n(t_m)$ of $Q_{j,n}$. This is due to condition \ref{C1}. As in the proof of Lemma~\ref{l:var_conv}, 
By the assumption $ns_n^d = 1$, one can easily show that $\Pn\big(\text{Tube}(Q_{j,n}, \, 2r_n(t_m))\big)$ is stochastically dominated by a Poisson random variable with positive and finite parameter, which does not depend on $j$ and $n$. Denote such a Poisson random variable by $Y$. Then, for $p=3,4$, 
$$
\max_j \E \Big[ \big| \xi_{j,n}-\E[\xi_{j,n}] \big|^p  \Big] \le C^*\Big[ \big( \E[2^{pY}] \big)^{1/p} + \E (2^Y)\Big] <\infty. 
$$
Referring back to \eqref{e:stein_bound} and noting $ns_n^d = 1$, we can see that 
$$
\Bigl|\,  \P\bigl( \sum_{j \in H_n} \barxi_{j,n} \leq x \bigr) - \P(Z \leq x) \, \Bigr| \le C^* \Big( \sqrt{s_n^{-d}n^{-3/2}}  + \sqrt{s_n^{-d}n^{-2}}\Big) = O(n^{-1/4}) \to 0, \ \ \ n\to\infty, 
$$
which implies that $\sum_{j\in H_n}\barxi_{j,n}\Rightarrow \mathcal N(0,1)$ as $n\to\infty$; equivalently, 
$$
\sum_{i=1}^m a_i \xnA (t_i) \Rightarrow \mathcal N (0,\Sigma_A),  \ \ \ n\to\infty, 
$$
where 
$$
\Sigma_A := \sum_{i=1}^m\sum_{j=1}^m a_i a_j \sum_{k_i=0}^{\infty} \sum_{k_j=0}^{\infty} (-1)^{k_i + k_j} \Psi_{k_i, k_j, A}(t_i, t_j). 
$$
Subsequently we claim that 
$$
\sum_{i=1}^m a_i \xn(t_i) \Rightarrow \mathcal N (0,\Sigma_{\R^d}),  \ \ \ n\to\infty, 
$$
which completes the proof. To show this, take $A_K=(-K,K)^d$ for $K>0$. It then suffices to verify that 
\begin{align*}
&\mathcal N(0,\Sigma_{A_K})\Rightarrow \mathcal N (0,\Sigma_{\R^d}), \ \ K \to\infty, 
\end{align*}
and for each $t\ge0$ and $\epsilon>0$, 
\begin{align*}
&\lim_{K\to\infty}\limsup _{n\to\infty} \P \Big( \big| \bar{\chi}_{n}(t)-  \bar{\chi}_{n,A_K}(t) \big| > \epsilon \Big)=0. 
\end{align*}
The former condition is obvious from $\Sigma_{A_K}\to \Sigma_{\R^d}$ as $K\to\infty$.
%the definition \eqref{e:def.Psi}, along with monotone convergence theorem. 
The latter is also a direct consequence of Proposition \ref{p:cov.asym}, together with Chebyshev's inequality and the fact that $\chi_n(t)-\chi_{n,A_K}(t)=\chi_{n,\R^d \setminus A_K}(t)$. 

\end{proof}

\subsection{Proof of tightness in Theorem \ref{t:main}}

Before we begin, a few more useful properties of $h_t^{k}$ are added. For $0 \le s <t <\infty$, we denote
$$
h_{t,s}^{k}(\Y) = h_t^{k}(\Y) - h_s^{k}(\Y), \ \ \Y=(y_0,\dots, y_k) \in (\R^d)^{k+1}. 
$$

\begin{lem} \label{l:hkl_cov} 
$(i)$ For any $0\le s \le t \le T < \infty$, 
$$
\int_{(\R^d)^k}h_{t,s}^{k}(0,y_1,\dots,y_k) \dif \by\le C_{d,k,T} (t^d-s^d), 
$$
where $C_{d,k,T}=k^2 (2^d\theta_d)^kT^{d(k-1)}$. \\

$(ii)$ Let $j \in \{1, \dots, (k_1 \wedge k_2)+1\}$ and suppose that $\by_0 \in (\R^d)^{j-1}$, $\by_1 \in (\R^d)^{k_1 + 1-j}$ and $\by_2 \in (\R^d)^{k_2+1-j}$. Then, for $0 \le t_1 \le s \le t_2 \le T < \infty$, 
\begin{align*}
&\int_{(\R^d)^{k_1 + k_2 + 1 - j}} h_{s,t_1}^{k_1} (0, \y_0, \y_1) h_{t_2, s}^{k_2}(0, \y_0, \y_2) \dif{\y_0} \dif{\y_1} \dif{\y_2} \\
&\phantom{\int_{(\R^d)^{k_1 + k_2 + 1 - j}} h_{s,t_\id \wedge  \one}^{k_1, 1} (0, \y_0, \y_1) h_{t_2, s}^{k_2, \id \wedge \one}} \leq 36(k_1 k_2)^6 ((2T)^d \theta_d)^{2(k_1 + k_2)} (t_2^d - t_1^d)^2. 
\end{align*}
\end{lem}
%%%%%%%%
\begin{proof}
%The relation \eqref{e:hkl_prop_1} is fairly obvious because 
%\[
%x_0 \in \bigcap_{y \in \sigma} B(y, t/c) \quad \Leftrightarrow \quad cx_0 + x \in \bigcap_{y \in \sigma} B(cy + x, t),
%\]
%for finite $\sigma \subset \R^d$, by the corresponding properties of Euclidean distance. The inequality \eqref{e:hkl_prop_2} is also fairly clear. To see so, note that $\ell(k) \geq 1$ by \ref{C1}, and hence if for every $\sigma \subset \Y$ such that $\card{\sigma} = \ell(k)+ 1$
%\[
%\bigcap_{y \in \sigma} B(y, t) \neq \emptyset
%\]
%holds, then it must hold for every $\tau \subset \sigma$ with $\tau = 2$. Hence $h^{k, \ell}_t(y_0, \dots, y_k) = 1$ implies $h^{k, 1}_t(y_0, \dots, y_k)$, a generalization of the well-known result $\C(\X, t) \subset \VR(\X, t)$ (see for instance \cite{chazalgeom}). Now we can prove \eqref{e:hkl_prop_3}. 
%\remove{By the changes of variable $y_i \mapsto y_i/t$ and $y_i \mapsto y_i/s$, and $\eqref{e:hkl_prop_2}$, we have
%\begin{align*}
%\int_{(\R^d)^k} h^{k}_{t, s}(0, y_1,\dots, y_k) \dif \by &= (t^{dk} - s^{dk})\int_{(\R^d)^k} h^{k, \ell}_1(0, y_1, \dots, y_k) \dif \by \\
%&\leq (t^{dk} - s^{dk})\int_{(\R^d)^k} h^{k, \id \wedge \one}_1(0, y_1, \dots, y_k) \dif \by \\
%&= \int_{(\R^d)^k} h^{k, \id \wedge \one}_{t, s}(0, y_1,\dots, y_k) \dif \by.
%\end{align*}}
%%%%%%%%%%%%%%%%%%%%%%%
We note that for any $0 \le s < t$ with $y_0\equiv 0$, 
\begin{align*}
&h_{t,s}^{k}(0, y_1, \dots, y_k) = \ind{ 2s < \max_{0 \leq i < j \leq k} \norm{y_i - y_j} \leq 2t} \\
&\qquad \leq \prod_{i=1}^k \one \big\{ y_i \in B(0,2T) \big\} \bigg( \sum_{i=1}^k \ind{ 2s < \norm{y_i} \leq 2t } + \sum_{1 \leq i < j \leq k} \ind{ 2s < \norm{y_i - y_j} \leq 2t} \bigg).
\end{align*}
For each $i = 1, \dots, k$, let $\mathbf{y}^{(i)}$ be the tuple $(y_1, \dots, y_{i-1}, y_{i+1}, \dots, y_k) \in (\R^d)^{k-1}$ with the $i$th coordinate omitted. Then, 
\begin{align*}
\int_{(\R^d)^k} h^{k}_{t,s}(0, y_1, \dots, y_k) \dif \by &\leq \sum_{i=1}^k \int_{B(0,2T)^{k-1}} \int_{\R^d} \ind{ 2s < \norm{y_i} \leq 2t } \dif{y_i} \dif{\mathbf{y}^{(i)}} \\
&+ \sum_{1 \leq i < j \leq k} \int_{B(0,2T)^{k-1}} \int_{\R^d} \ind{ 2s < \norm{y_{i} - y_j} \leq 2t } \dif{y_i} \dif{\mathbf{y}^{(i)}}. \\
&= \Big( k+\binom{k}{2} \Big) m\big( B(0,2T) \big)^{k-1} \big[ m\big(B(0,2t)\big)-m\big(B(0,2s)\big) \big] \\
&\le C_{d,k,T}(t^d-s^d)
\end{align*}
as required. 

Part $(ii)$ is essentially the same as Lemma 7.1 in \cite{owadaFCLT}, so the proof is skipped. 
\end{proof}

%%%%%%%%%%%%%%%%%%%%%%%%%%%%%%%%%%%%%%%%%%%%%%%%%%%%%%%%%%%%%%%
%%%%%%%%%%%%%%%%%%%%%%%%%%%%%%%%%%%%%%%%%%%%%%%%%%%%%%%%%%%%%%%
%%%%%%%%%%%%%%%%%%%%%%%%%%%%%%%%%%%%%%%%%%%%%%%%%%%%%%%%%%%%%%%
%%%%%%%%%%%%%%%%%%%%%%    T I G H T N E S S    P R O O F   %%%%%%%%%%%%%%%%%%%%%%%
%%%%%%%%%%%%%%%%%%%%%%%%%%%%%%%%%%%%%%%%%%%%%%%%%%%%%%%%%%%%%%%
%%%%%%%%%%%%%%%%%%%%%%%%%%%%%%%%%%%%%%%%%%%%%%%%%%%%%%%%%%%%%%%
%%%%%%%%%%%%%%%%%%%%%%%%%%%%%%%%%%%%%%%%%%%%%%%%%%%%%%%%%%%%%%%
\begin{proof}[Proof of tightness in Theorem \ref{t:main}]
To show tightness, it suffices to use Theorem 13.5 from \cite{Billingsley} that requires that for every $0<T<\infty$, there exists a $C > 0$ such that
\begin{align} \label{e:mod_cond}
\E\big[ |\xn(t_2) - \xn(s)|^2 |\xn(s) - \xn(t_1)|^2 \big] \leq C(t_2^d - t_1^d)^2,
\end{align}
for all $0 \leq t_1 \leq s \leq t_2 \leq T$ and $n\in \Nat$. 
%\[
%\mathcal{H}^1(T - \delta) \xrightarrow{P} \mathcal{H}^1(T), \quad \delta \to 0.
%\]
%Theorem 16.7 in the same reference (\cite{Billingsley}) furnishes the desired convergence in $D[0, \infty)$. The second condition is satisfied immediately by Proposition \ref{p:cont_proc}, which implies continuity in probability. Now, all that remains is the non-trivial task of demonstrating that the first condition holds. 
To demonstrate \eqref{e:mod_cond}, we will give an abridged proof -- tightness will be similarly established for analogous processes seen in \cite{owadaFCLT, penrose2000}. Let us begin with some helpful notation, namely,
\begin{align*}
&h_{n, t, s}^{k}(\Y) := h_{r_n(t), r_n(s)}^{k}(\Y) = h_{r_n(t)}^k (\Y) - h_{r_n(s)}^k (\Y), \\
&\zeta^{k}_{n, t, s} := S_k\big(\Pn, r_n(t)\big) - S_k\big(\Pn, r_n(s)\big) = \sum_{\Y \subset \Pn} h_{n, t, s}^{k}(\Y). 
\end{align*}
%for which, as usual we have omitted the superscript $\id \wedge \one$ from the indicators and other objects. 
%In the ensuing, we omit the superscript 1, for conciseness of notation. Therefore, we have for any $t, s \geq 0$ and 
%\begin{equation} \label{e:simp_fubini}
%\sum_{k=0}^{\Pn(\R^d)} S_k(\Pn, r_n(t)) \leq 2^{\Pn(\R^d)}, 
%\end{equation}
%which is a.s. finite, that 
%\[
%\chi^1_n(t) - \chi^1_n(s) = \sum_{k=0}^{\infty} (-1)^k \zeta^{(k)}_{n, t, s} \quad \text{a.s.}
%\]
%Similarly, we obtain
%\[
%\xto_n(t) - \xto_n(s) = n^{-1/2}\sum_{k=0}^{\infty} (-1)^k \big(\zeta^{(k)}_{n, t, s} - \E[\zeta^{(k)}_{n, t, s}]\big).
%\]
%Thus,
%\begin{align*}
%|\xto_n(t) - \xto_n(s)|^2 &\leq n^{-1} \bigg(\sum_{k=0}^{\infty} (-1)^k\big(\zeta^{(k)}_{n, t, s} - \E[\zeta^{(k)}_{n, t, s}]\big ) \bigg)^2 \\
%&=  n^{-1} \sum_{k_1=0}^{\infty} \sum_{k_2=0}^{\infty} (-1)^{k_1 + k_2} \big(\zeta^{(k_1)}_{n, t, s} - \E[\zeta^{(k_1)}_{n, t, s}]\big ) \big(\zeta^{(k_2)}_{n, t, s} - \E[\zeta^{(k_2)}_{n, t, s}]\big ),
%\end{align*}
%by Fubini's theorem. The justification as to why we can use Fubini's theorem is due to repeated applications of the triangle inequality, monotonicity of $S_k(\Pn, r_n(t))$ in $t$ and the bound \eqref{e:simp_fubini}.

By the same argument as in the proof of Proposition \ref{p:cov.asym}, one can apply Fubini's theorem to obtain 
%we conclude
\begin{align}
&\E\big[ |\xn(t_2) - \xn(s)|^2 |\xn(s) - \xn(t_1)|^2 \big]  \label{e:tightness.4th} \\
&\qquad= \frac{1}{n^2} \sum_{(k_1, k_2, k_3, k_4) \in \Nat_0^4} (-1)^{k_1 + k_2 + k_3 + k_4} \E\Big[ \big(\zeta^{k_1}_{n, t_2, s} - \E[\zeta^{k_1}_{n, t_2, s}]\big ) \big(\zeta^{k_2}_{n, t_2, s} - \E[\zeta^{k_2}_{n, t_2, s}]\big ) \notag \\
&\phantom{\qquad= \frac{1}{n^2} \sum_{\mathbf{k} \in \Nat_0^4} (-1)^{k_1 + k_2 + k_3 + k_4} \big(\zeta^{(k_1)}_{n, s, t_1} - \E[\zeta^{(k_1)}} \times \big(\zeta^{k_3}_{n, s, t_1} - \E[\zeta^{k_3}_{n, s, t_1}]\big ) \big(\zeta^{k_4}_{n, s, t_1} - \E[\zeta^{k_4}_{n, s, t_1}]\big ) \Big]. \notag %PROBABLY NEED TO SAY MORE ABOUT THIS, CERTAINLY NON-TRIVIAL BUT REPEATING THE SAME ARGUMENT AS SEEN SO OFTEN BEFORE IS CERTAINLY TEDIOUS (7/22/19).  
\end{align}
Our objective now is to find a suitable bound for 
\begin{align} \label{e:big_sum}
\E\Big[ \big(\zeta^{k_1}_{n, t_2, s} - \E[\zeta^{k_1}_{n, t_2, s}]\big ) \big(\zeta^{k_2}_{n, t_2, s} - \E[\zeta^{k_2}_{n, t_2, s}]\big ) \big(\zeta^{k_3}_{n, s, t_1} - \E[\zeta^{k_3}_{n, s, t_1}]\big ) \big(\zeta^{k_4}_{n, s, t_1} - \E[\zeta^{k_4}_{n, s, t_1}]\big ) \Big]. 
\end{align}
To this end, let us refine the notation once more by denoting $\xi_1 := \zeta^{k_1}_{n, t_2, s}$, $\xi_2 := \zeta^{k_2}_{n, t_2, s}$, $\xi_3 := \zeta^{k_3}_{n, s, t_1}$ and $\xi_4 := \zeta^{k_4}_{n, s, t_1}$. Furthermore, let $h_1 := h_{n, t_2, s}^{k_1}$, $h_2 := h_{n, t_2, s}^{k_2}$, $h_3 := h_{n, s, t_1}^{k_3}$ and $h_4 := h_{n, s, t_1}^{k_4}$. Define $[n] := \{1,2, \dots, n\}$ and for any $\sigma \subset [4]$ let $\xi_\sigma = \prod_{i \in \sigma} \xi_i$ where we set $\xi_\emptyset = 1$ by convention. Then we can express \eqref{e:big_sum} quite simply as
\begin{align} \label{e:xisum}
\sum_{\sigma \subset [4]} (-1)^{|\sigma|}\,  \E[ \xi_\sigma] \prod_{i \in [4]\setminus \sigma} \E[\xi_i]. 
\end{align}
For $\sigma \subset [4]$ with $\sigma \neq \emptyset$, and finite subsets $\Y_j \subset \R^d$, $j\in \sigma$, we define $\Y_\sigma := \bigcup_{j \in \sigma}\Y_j$. 
Given a subset $\tau \subset \sigma \subset [4]$, we also define 
\begin{align*}
\mathcal I_{\tau, \sigma} (\Y_\sigma) &:= \prod_{j\in \tau} \one \big\{ \text{there exists } p \in \tau\setminus \{j\} \text{ such that } \Y_j \cap \Y_p \neq \emptyset  \big\} \\
&\quad \times \prod_{j \in \sigma \setminus \tau} \one \big\{ \Y_j \cap \Y_q =\emptyset \text{ for all } q \in \sigma \setminus \{j\}  \big\}. 
\end{align*}
Note that $\mathcal I_{\tau, \sigma}(\Y_\sigma)= 0$ whenever $|\tau|=1$, and 
\begin{equation}  \label{e:sum.I.1}
\sum_{\tau \subset \sigma} \mathcal I_{\tau, \sigma}(\Y_\sigma)=1. 
\end{equation}
Furthermore, if $\tau=\sigma$, we write $\mathcal I_\sigma(\cdot):= \mathcal I_{\sigma, \sigma}(\cdot)$. It follows from \eqref{e:sum.I.1} and Palm theory in the Appendix that, for each non-empty $\sigma \subset [4]$, 
\begin{align*}
\E[\xi_\sigma] &= \E \Big[ \sum_{\Y_j \subset \Pn, \, j\in \sigma} \prod_{i \in \sigma} h_i (\Y_i) \Big] \\
&= \sum_{\tau \subset \sigma} \E \Big[  \sum_{\Y_j \subset \Pn, \, j\in \sigma}  \mathcal I_{\tau, \sigma} (\Y_\sigma) \prod_{i\in \sigma} h_i(\Y_i) \Big] \\
&=\sum_{\tau \subset \sigma} \E \Big[  \sum_{\Y_j \subset \Pn, \, j\in \tau}  \mathcal I_{\tau} (\Y_\tau)  \prod_{i\in \tau} h_i(\Y_i) \Big] \prod_{i \in \sigma \setminus \tau} \E[\xi_i]. 
\end{align*}
Hence, \eqref{e:xisum} is equal to 
\begin{align*}
&\sum_{\sigma \subset [4]} \sum_{\tau \subset \sigma} (-1)^{|\sigma|}\, \E \Big[  \sum_{\Y_j \subset \Pn, \, j\in \tau}  \mathcal I_{\tau} (\Y_\tau)\prod_{i\in \tau} h_i(\Y_i)  \Big] \prod_{i \in \sigma \setminus \tau} \E[\xi_i] \prod_{i \in [4]\setminus \sigma} \E[\xi_i] \\
&= \sum_{\tau \subset [4]} \E \Big[  \sum_{\Y_j \subset \Pn, \, j\in \tau}  \mathcal I_{\tau} (\Y_\tau)\prod_{i\in \tau} h_i(\Y_i)  \Big] \prod_{i \in [4]\setminus \tau} \E[\xi_i] \sum_{\tau \subset \sigma \subset [4]} (-1)^{|\sigma|} \\
&= \E \Big[  \sum_{\Y_1\subset \Pn} \sum_{\Y_2\subset \Pn} \sum_{\Y_3\subset \Pn} \sum_{\Y_4\subset \Pn} \mathcal I_{[4]}(\Y_{[4]}) \prod_{i=1}^4 h_i (\Y_i) \Big], 
\end{align*}
where the last line follows from the fact that $\sum_{\tau \subset \sigma \subset [4]} (-1)^{|\sigma|} = \binom{4-|\tau|}{0}(-1)^{|\tau|} + \dots + \binom{4-|\tau|}{4-|\tau|}(-1)^4 = 0$, unless $\tau = [4]$. 
Substituting this back into \eqref{e:tightness.4th} and taking the absolute value of $(-1)^{k_1+k_2+k_3+k_4}$, we get 
\begin{align*}
&\E\big[ |\xn(t_2) - \xn(s)|^2 |\xn(s) - \xn(t_1)|^2 \big] \\
&\qquad \leq \sum_{(k_1, k_2, k_3, k_4) \in \Nat_0^4} \frac{1}{n^2} \E \Big[ \sum_{\Y_1 \subset \Pn} \sum_{\Y_2 \subset \Pn} \sum_{\Y_3 \subset \Pn}  \sum_{\Y_4 \subset \Pn} \mathcal I_{[4]}(\Y_{[4]}) \prod_{i = 1}^4 h_i(\Y_{i}) \Big]. 
\end{align*}  

Now, it suffices to show that the right-hand side above is less than $C(t_2^d-t_1^d)^2$ for some $C > 0$. We can break the above summand into four distinct cases: 
\begin{itemize}
\item[\textbf{(I)}] $b_{12}=|\Y_1 \cap \Y_2|  >0$, $b_{34}=|\Y_3 \cap \Y_4| >0$ with all other pairwise intersections empty. 
\item[\textbf{(II)}] $b_{13}=|\Y_1 \cap \Y_3| >0$, $b_{24}=|\Y_2 \cap \Y_4| >0$ with all other pairwise intersections empty. 
\item[\textbf{(III)}] $b_{14}=|\Y_1 \cap \Y_4|  >0$, $b_{23}=|\Y_2 \cap \Y_3| >0$ with all other pairwise intersections empty. 
\item[\textbf{(IV)}] For each $i$, there exists a $j\neq i$ such that $ \Y_i \cap \Y_j \neq \emptyset$, but $\textbf{(I)-(III)}$ do not hold. 
\end{itemize}
We prove appropriate upper bounds for cases \textbf{(I)} and \textbf{(IV)}, and the other two cases follow from the proof for \textbf{(I)}. Palm theory in the Appendix implies that
\begin{align} \label{e:case1}
&\frac{1}{n^2} \E \Big[ \sum_{\Y_1 \subset \Pn} \sum_{\Y_2 \subset \Pn}  \sum_{\Y_3 \subset \Pn}  \sum_{\Y_4 \subset \Pn} \prod_{i=1}^4 h_i(\Y_i) \\
&\phantom{\E \Big[ \sum_{\Y_1 \subset \Pn} \cdots} \times \ind{ |\Y_1 \cap \Y_2| = b_{12}, |\Y_3 \cap \Y_4| = b_{34}, |\Y_i \cap \Y_j| = 0 \text{ for other } i,j\text{'s}} \Big] \notag \\
&= \frac{1}{n^2} \E\Big[ \sum_{\Y_1 \subset \Pn} \sum_{\Y_2 \subset \Pn} h_1(\Y_1) h_2(\Y_2) \ind{ |\Y_1 \cap \Y_2| = b_{12}}\Big] \notag \\
&\phantom{ \frac{1}{n^2} \E\Big[ \sum_{\Y_1 \subset \Pn} \sum_{\Y_2 \subset \Pn}} \times \E\Big[ \sum_{\Y_3 \subset \Pn} \sum_{\Y_4 \subset \Pn} h_3(\Y_3) h_4(\Y_4)  \ind{ |\Y_3 \cap \Y_4| = b_{34}}\Big] \notag \\
&=\frac{n^{k_1 + k_2 + 1 - b_{12}}}{b_{12}! (k_1 + 1 - b_{12})! (k_2 + 1 - b_{12})!} \notag \\
&\phantom{n^{k_1 + k_2 + 1 - b_{12}}} \times \E\big[ h_1(X_1, \dots, X_{k_1 + 1}) h_2(X_1, \dots, X_{b_{12}}, X_{k_1 + 2}, \dots, X_{k_1 + k_2 + 2 - b_{12}})\big] \notag \\
&\times\frac{n^{k_3 + k_4 + 1 - b_{34}}}{b_{34}! (k_3 + 1 - b_{34})! (k_4 + 1 - b_{34})!} \notag \\
&\phantom{n^{k_1 + k_2 + 1 - j_{12}}} \times \E\big[ h_3(X_1, \dots, X_{k_3 + 1}) h_4(X_1, \dots, X_{b_{34}}, X_{k_3 + 2}, \dots, X_{k_3 + k_4 + 2 - b_{34}})\big]. \notag
\end{align}
%Denote $\y_0 = (y_1, \dots, y_{j_{12} - 1})$, $\y_1 = (y_{j_{12}}, \dots, y_{k_1})$ and $\y_2 = (y_{k_1 + 1}, \dots, y_{k_1 + k_2 + 1 - j_{12}})$. 

For the remainder of the proof, assume that $(2T)^d \theta_d > 1$, $\lVert f \rVert_{\infty} > 1$ and $T > 1$ for ease of description. %This is by no means necessary (we may take supremums in $j$), but simplifies the following argument. Write $j \equiv j_{12}$ and assume that $k_2 \geq k_1$. 
Moreover, assume, without loss of generality that $k_1 \ge k_2$ and $k_3 \ge k_4$. 
Using trivial bounds and the customary changes of variable, i.e., $x_1 = x$, $x_i = x + s_ny_{i-1}$, $i=2,\dots,k_1+k_2+2-b_{12}$, and applying Lemma \ref{l:hkl_cov} $(i)$, and recalling $a_T=(2T)^d \theta_d \|f\|_\infty$, we see that
\begin{align*}
&n^{k_1 + k_2 + 1 - b_{12}} \E[ h_1(X_1, \dots, X_{k_1 + 1}) h_2(X_1, \dots, X_{b_{12}}, X_{k_1 + 2}, \dots, X_{k_1 + k_2 + 2 - b_{12}})] \\
&\leq ( \lVert f \rVert_{\infty})^{k_1 + k_2 + 1 - b_{12}} 
\int_{(\R^d)^{k_2+1-b_{12}}}\int_{(\R^d)^{k_1+1-b_{12}}}\int_{(\R^d)^{b_{12}-1}}h_{t_2, s}^{k_1}(0, \y_0, \y_1) \\
&\qquad \qquad\qquad\qquad\qquad\qquad\qquad\qquad\qquad\qquad\qquad  \times h_{t_2, s}^{k_2}(0, \y_0, \y_2) \dif{\y_0} \dif{\y_1} \dif{\y_2} \\
&\leq ( \lVert f \rVert_{\infty})^{k_1 + k_2} \big((2T)^d\theta_d\big)^{k_2+1-b_{12}} \int_{(\R^d)^{k_1+1-b_{12}}}\int_{(\R^d)^{b_{12}-1}} h_{t_2, s}^{k_1}(0, \y_0, \y_1) \dif{\y_0} \dif{\y_1} \\
&\leq ( \lVert f \rVert_{\infty})^{k_1 + k_2} \big((2T)^d\theta_d\big)^{k_2+1-b_{12}} C_{d, k_1, T} (t_2^d - s^d) \\
&\le k_1^2 (a_T )^{\kokt}(t_2^d - s^d).
\end{align*} 
%Now, by the assumptions on $T$, we have that 
%\[
%((2T)^d\theta_d)^{k_2+1-j} C_{d, k_1, T} \leq k_1^2((2T)^d\theta_d)^{k_1 + k_2}.
%\]
Hence, \eqref{e:case1} is bounded by
\begin{align*}
&\frac{(a_T)^{k_1 + k_2}k_1^2 }{b_{12}!(k_1 + 1 - b_{12})! (k_2 + 1 - b_{12})!} (t_2^d - s^d) \frac{( a_T)^{k_3 + k_4}k_3^2 }{b_{34}!(k_3 + 1 - b_{34})! (k_4 + 1 - b_{34})!} (s^d - t_1^d) \\
&\leq \frac{(a_T)^{k_1 + k_2 + k_3 + k_4}k_1^2 k_3^2 }{b_{12}!(k_1 + 1 - b_{12})! (k_2 + 1 - b_{12})! b_{34}!(k_3 + 1 - b_{34})! (k_4 + 1 - b_{34})!} (t_2^d - t_1^d)^2. 
\end{align*}
Finally we see that
\begin{align*}
&\sum_{\substack{k_1 \geq k_2, \, k_3 \geq k_4, \\ 1 \leq b_{12} \leq k_2+1, \\ 1 \leq b_{34} \leq k_4+1}} \frac{(a_T)^{k_1 + k_2 + k_3 + k_4}k_1^2 k_3^2}{b_{12}!(k_1 + 1 - b_{12})! (k_2 + 1 - b_{12})! b_{34}!(k_3 + 1 - b_{34})! (k_4 + 1 - b_{34})!} <\infty, 
\end{align*}
since 
\begin{align}
\sum_{k_1=0}^\infty \sum_{k_2=0}^{k_1} \sum_{\ell =1}^{k_2+1} \frac{(a_T)^{k_1+k_2}k_1^2}{\ell! (k_1+1-\ell)! (k_2+1-\ell)!}  &= \sum_{\ell=1}^\infty \sum_{k_1=\ell-1}^\infty \frac{(a_T)^{k_1}k_1^2}{\ell! (k_1+1-\ell)!}\, \sum_{k_2=\ell-1}^{k_1} \frac{(a_T)^{k_2}}{(k_2+1-\ell)!}  \notag \\
&\le e^{a_T}\sum_{\ell=1}^\infty \frac{(a_T)^{\ell-1}}{\ell!}  \sum_{k_1=\ell-1}^\infty \frac{(a_T)^{k_1}k_1^2}{(k_1+1-\ell)!}\, <\infty.  \label{e:summable}
\end{align}
%&\ed{\leq \sum_{k_1=0}^{\infty} \sum_{k_3 = 0}^{\infty}  \sum_{k_2 = 0}^{k_1}  \sum_{k_4 = 0}^{k_3} \frac{(a_T)^{k_1 + k_2 + k_3 + k_4}%k_1^2 k_3^2 2^{k_2} 2^{k_4}}{(k_1-k_2)!k_2! (k_3-k_4)! k_4!}  < \infty.}
Now, for cases \textbf{(I)} - \textbf{(III)}, we have an upper bound of the form $C(t_2^d - t_1^d)^2$ as desired. 

Thus we need only demonstrate the same for case \textbf{(IV)}. In addition to the notation $b_{ij}$, $1\le i < j \le 4$ as above, define for $\Y_i \in (\R^d)^{k_i+1}$, $k_i \in \Nat_0$, $i=1,\dots,4$, 
$$
b_{ijk} := |\Y_i \cap \Y_j \cap \Y_k|, \ 1\le i < j < k \le 4, \ \ \text{and }\,  b_{1234} := |\Y_1 \cap \Y_2 \cap \Y_3 \cap \Y_4|, 
$$
and also, 
\begin{equation}  \label{e:def.b}
b := b_{12} + b_{13} + b_{14} + b_{23} + b_{24} + b_{34} - b_{123} - b_{124} - b_{134}  -b_{234} + b_{1234}, \\
\end{equation} 
so that $|\Y_1 \cup \Y_2 \cup \Y_3 \cup \Y_4| = k_1 +k_2 +k_3 +k_4+4-b$ with $b\ge 3$. 
Let $\B$ be the collection of $\bb=(b_{12}, \dots, b_{1234}) \in \Nat_0^{11}$ satisfying the conditions in Case \textbf{(IV)}. 
For a non-empty $\sigma \subset [4]$ and $\Y_i \in (\R^d)^{k_i+1}$, $i=1,\dots,4$, let 
\begin{equation}  \label{e:def.j.sigma}
j_\sigma := \biggl| \,  \bigcap_{i\in \sigma} \Big( \Y_i \setminus \bigcup_{j\in [4]\setminus \sigma} \Y_j  \Big) \, \bigg|
\end{equation}
In particular, $j_\sigma$'s are functions of $\bb$ such that $\sum_{\sigma \subset [4], \, \sigma \neq \emptyset} j_\sigma = |\Y_1 \cup \Y_2 \cup \Y_3 \cup \Y_4|$. Then, Palm theory in the Appendix yields 
\begin{align*}
&\frac{1}{n^2}\, \E \Big[ \sum_{\Y_1 \subset \Pn} \sum_{\Y_2 \subset \Pn} \sum_{\Y_3 \subset \Pn}  \sum_{\Y_4 \subset \Pn} \prod_{i=1}^4h_i(\Y_i)\, \ind{ \text{case \textbf{(IV)} holds}} \Big] \\
&=\sum_{\bb \in \B} \frac{1}{n^2}\, \E \Big[ \sum_{\Y_1 \subset \Pn}  \sum_{\Y_2 \subset \Pn} \sum_{\Y_3 \subset \Pn}  \sum_{\Y_4 \subset \Pn} \prod_{i=1}^4h_i(\Y_i)\,  \one\big\{  |\Y_1\cap \Y_2|=b_{12}, |\Y_1 \cap \Y_3|=b_{13},  \\
&\qquad \qquad \qquad \qquad \qquad \qquad \qquad \qquad \dots, |\Y_1 \cap \Y_2 \cap \Y_3 \cap \Y_4|=b_{1234} \big\} \Big] \\
&=\sum_{\bb \in \B}  \frac{n^{k_1+k_2+k_3+k_4+2-b}}{\prod_{\sigma \subset [4], \, \sigma \neq \emptyset} j_\sigma!}\, 
\E \Big[ \prod_{i=1}^4h_i(\Y_i)\,  \one\big\{  |\Y_1\cap \Y_2|=b_{12}, |\Y_1 \cap \Y_3|=b_{13},   \\
&\qquad \qquad \qquad \qquad \qquad \qquad \qquad \qquad \dots, |\Y_1 \cap \Y_2 \cap \Y_3 \cap \Y_4|=b_{1234} \big\} \Big]. 
\end{align*}

Under the conditions in Case \textbf{(IV)}, at least one of the $b_{ij}$'s is non-zero, so we may assume without loss of generality that $b_{13}>0$. Then we have 
\begin{align*}
&n^{k_1+k_2+k_3+k_4+2-b} \E \Big[ \prod_{i=1}^4h_i(\Y_i)\,  \one\big\{  |\Y_1\cap \Y_2|=b_{12},  \dots, |\Y_1 \cap \Y_2 \cap \Y_3 \cap \Y_4|=b_{1234} \big\} \Big]  \\
&= n^{k_1+k_2+k_3+k_4+2-b} \int_{(\R^d)^{k_1+k_2+k_3+k_4+4-b}} h_1(\bx_0, \bx_1) h_3(\bx_0, \bx_3) h_2(\bx_2) h_4 (\bx_4) \\
&\qquad \qquad \qquad \qquad \qquad \qquad \qquad \qquad \qquad \qquad \times  \prod_{x \in \bigcup_{i=0}^4 \bx_i} f(x) \dif \, (\bx_0 \cup \bx_1 \cup \cdots \cup \bx_4), 
\end{align*}
where $\bx_0$ is a collection of elements in $\R^d$ with $|\bx_0|=b_{13}>0$. In other words, $\bx_0 \in (\R^d)^{b_{13}}$, so that 
$\bx_1\in (\R^d)^{k_1+1-b_{13}}$ and $\bx_3\in (\R^d)^{k_3+1-b_{13}}$ with $\bx_1 \cap \bx_3 = \emptyset$. 
Moreover, $\bx_2 \in (\R^d)^{k_2+1}$ and $\bx_4\in (\R^d)^{k_4+1}$, such that if $\bx_2 \cap \bx_4 =\emptyset$, then $\bx_i \cap (\bx_0 \cup \bx_1 \cup \bx_3) \neq \emptyset$ for $i=2,4$, and if $\bx_2 \cap \bx_4 \neq \emptyset$, then $(\bx_2 \cup \bx_4) \cap (\bx_0 \cup \bx_1 \cup \bx_3)\neq \emptyset$. 

Now, let us perform a change of variables by $\bx_i = x\one +s_n \by_i$ for $i=0,\dots,4$, where $\one$ is a vector with all entries $1$, and the first element of $\by_0$ is taken to be $0$. In addition to this, we apply the translation and scale invariance of $h_i$'s to get 
\begin{align*}
&n^{k_1+k_2+k_3+k_4+2-b} \E \Big[ \prod_{i=1}^4h_i(\Y_i)\,  \one\big\{  |\Y_1\cap \Y_2|=b_{12},  \dots, |\Y_1 \cap \Y_2 \cap \Y_3 \cap \Y_4|=b_{1234} \big\} \Big]  \\
&=n^{k_1+k_2+k_3+k_4+2-b}s_n^{d(k_1+k_2+k_3+k_4+3-b)} \int_{\R^d} \int_{(\R^d)^{k_1+k_2+k_3+k_4+3-b}} \hspace{-10pt} h_{t_2,s}^{k_1}(\by_0, \by_1) h_{s, t_1}^{k_3}(\by_0, \by_3) \\
&\qquad \qquad \qquad \qquad \times h_{t_2,s}^{k_2}(\by_2) h_{s, t_1}^{k_4}(\by_4) \prod_{y \in \bigcup_{i=0}^4 \by_i} f(x+s_n y) \dif \big(  (\by_0 \cup \dots \cup \by_4)\setminus \{ 0 \} \big)  \dif x. 
\end{align*}
Using $ns_n^d=1$, together with the trivial bounds $h_{t_2,s}^{k_2}(\by_2) \le h_T^{k_2}(\by_2)$, $h_{s, t_1}^{k_4}(\by_4)\le h_T^{k_4}(\by_4)$, and $f(x+s_n y) \le \| f \|_\infty$, one can bound the last expression by 
\begin{align}
&\| f \|_\infty^{k_1+k_2+k_3+k_4+3-b} \int_{(\R^d)^{k_1+k_2+k_3+k_4+3-b}} \hspace{-10pt} h_{t_2,s}^{k_1}(\by_0, \by_1) h_{s, t_1}^{k_3}(\by_0, \by_3) \notag \\
&\qquad \qquad \qquad \qquad \qquad \qquad \qquad \qquad \qquad \times  h_T^{k_2}(\by_2) h_T^{k_4}(\by_4) \dif \big(  (\by_0 \cup \dots \cup \by_4)\setminus \{ 0 \} \big) \notag  \\
&= \| f \|_\infty^{k_1+k_2+k_3+k_4+3-b} \int_{(\R^d)^{k_1+k_3+1-b_{13}}} \hspace{-10pt} h_{t_2,s}^{k_1}(\by_0, \by_1) h_{s, t_1}^{k_3}(\by_0, \by_3) \notag \\
&\qquad \quad \times \bigg\{ \int_{(\R^d)^{k_2+k_4+2-b+b_{13}}}  \hspace{-10pt} h_T^{k_2}(\by_2) h_T^{k_4}(\by_4)  \dif \big( (\by_2 \cup \by_4) \setminus (\by_0 \cup \by_1 \cup \by_3)  \big) \bigg\} \dif \big( \by_0\setminus \{ 0 \} \big)\dif \by_1 \dif \by_3.  \label{e:tight.bound.integral}
\end{align}
Suppose $h_T^{k_2}(\by_2) h_T^{k_4}(\by_4) =1$, such that 
\begin{equation}  \label{e:cond.y2.y4}
\by_2 \cap \by_4 \neq \emptyset, \ \ \by_2 \cap (\by_0 \cup \by_1 \cup \by_3) = \emptyset, \ \ \by_4 \cap (\by_0 \cup \by_1 \cup \by_3) \neq \emptyset. 
\end{equation}
Then, there exists $y'\in \by_4 \cap (\by_0 \cup \by_1 \cup \by_3)$ such that all points in $\by_2$ are at distance at most $4T$ from $y'$. 
Since $y'$ itself lies within distance $2T$ from the origin (recall that the first element of $\by_0$ is $0$), we conclude that all points in $\by_2 \cap \by_4$ are at distance at most $6T$ from the origin. As $b_{13} \le k_1 + k_3 +1$ and $b \ge 3$, we have 
\begin{align}
&\int_{(\R^d)^{k_2+k_4+2-b+b_{13}}}  \hspace{-10pt} h_T^{k_2}(\by_2) h_T^{k_4}(\by_4)  \dif \big( (\by_2 \cup \by_4) \setminus (\by_0 \cup \by_1 \cup \by_3)  \big) \label{e:distance.6T}\\
&\le m\big( B(0,6T) \big)^{k_2+k_4+2-b+b_{13}} = \big( (6T)^d \theta_d \big)^{k_2+k_4+2-b+b_{13}} \le \big( (6T)^d \theta_d \big)^{k_1+k_2+k_3 + k_4}. \notag
\end{align}
If $\by_2$ and $\by_4$ do not satisfy \eqref{e:cond.y2.y4}, it is still easy to check \eqref{e:distance.6T}. 

Applying \eqref{e:distance.6T}, along with Lemma \ref{l:hkl_cov} $(ii)$, one can bound \eqref{e:tight.bound.integral} by 
\begin{align*}
&\| f \|_\infty^{k_1+k_2+k_3+k_4+3-b}  \big( (6T)^d \theta_d \big)^{k_1+k_2+k_3 + k_4} \times 36 (k_1k_3)^6 \big( (2T)^d \theta_d \big)^{2(k_1+k_3)} (t_2^d -t_1^d)^2 \\
&\le 36 (k_1 k_2 k_3 k_4)^6 \big( (6T)^d \theta_d \| f \|_\infty \big)^{3(k_1 + k_2 + k_3+k_4)} (t_2^d -t_1^d)^2. 
\end{align*}

Thus, we conclude that 
\begin{align*}
&\frac{1}{n^2}\, \E \Big[ \sum_{\Y_1 \subset \Pn} \sum_{\Y_2 \subset \Pn} \sum_{\Y_3 \subset \Pn}\sum_{\Y_4 \subset \Pn} \prod_{i=1}^4h_i(\Y_i)\, \ind{ \text{case \textbf{(IV)} holds}} \Big]  \\
&\le 36 (k_1 k_2 k_3 k_4)^6 \big( (6T)^d \theta_d \| f \|_\infty \big)^{3(k_1 + k_2 + k_3+k_4)} \sum_{\bb\in \B} \frac{1}{\prod_{\sigma \subset [4], \, \sigma \neq \emptyset}j_\sigma !}\, (t_2^d -t_1^d)^2. 
\end{align*}
To complete the proof, we need to show that 
$$
\sum_{k_1 \le k_2 \le k_3 \le k_4}  (k_1 k_2 k_3 k_4)^6 \big( (6T)^d \theta_d \| f \|_\infty \big)^{3(k_1 + k_2 + k_3+k_4)} \sum_{\bb\in \B} \frac{1}{\prod_{\sigma \subset [4], \, \sigma \neq \emptyset}j_\sigma !} <\infty. 
$$
As seen in the calculation at \eqref{e:summable}, the term $(k_1 k_2 k_3 k_4)^6 \big( (6T)^d \theta_d \| f \|_\infty \big)^{3(k_1 + k_2 + k_3+k_4)}$ is negligible,  while proving
$$
\sum_{k_1 \le k_2 \le k_3 \le k_4}  \sum_{\bb\in \B} \frac{1}{\prod_{\sigma \subset [4], \, \sigma \neq \emptyset}j_\sigma !} <\infty
$$
is straightforward.

\end{proof}

\subsection{Proof of H\"older continuity of $\mathcal H$}
\begin{proof}[Proof of H\"older continuity in Theorem \ref{t:main}]
Since $\mathcal H(t)-\mathcal H(s)$ has a normal distribution for $0 \le s < t<\infty$, we have for every $m\in \Nat$, 
$$
\E\Big[ \big( \mathcal{H}(t) - \mathcal{H}(s) \big)^{2m} \Big] = \prod_{i=1}^m (2i-1) \Big( \E\big[ \big( \mathcal{H}(t) - \mathcal{H}(s) \big)^{2} \big] \Big)^m,
$$
Proposition \ref{p:cov.asym} ensures that $\big( \sum_{k=0}^M(-1)^k \mathcal H_k(t), \, M\in \Nat_0 \big)$ constitutes a Cauchy sequence in $L^2(\Omega)$. Therefore we have
\begin{align*}
\E\big[ \big( \mathcal{H}(t) - \mathcal{H}(s) \big)^{2}\big] &=\lim_{M\to\infty} \E \bigg[ \Big( \sum_{k=0}^M (-1)^k \big(\mathcal H_k(t)-\mathcal H_k(s) \big) \Big)^2 \bigg] \\
&\le \bigg[ \sum_{k=0}^\infty \bigg\{ \E \Big[ \big( \mathcal H_k(t) - \mathcal H_k(s) \big)^2 \Big]  \bigg\}^{1/2}  \bigg]^2,
\end{align*}
where the second line is due to the Cauchy-Schwarz inequality. 
We see at once that
\begin{align}
\E\big[ \big( \mathcal{H}_k(t) - \mathcal{H}_k(s) \big)^{2} \big] &= \Psi_{k,k}(t, t) - 2\Psi_{k,k}(t, s) + \Psi_{k,k}(s,s) \notag\\
&\le  \Psi_{k,k}(t, t)  - \Psi_{k,k}(t, s) =\sum_{j=1}^{k+1} \big( \psi_{j,k,k}(t,t) - \psi_{j,k,k}(t,s) \big) \label{e:last.step}
\end{align} 
by monotonicity due to \eqref{e:monotonicity.indicator} and symmetry of $\Psi_{k,k}(\cdot, \cdot)$ in its arguments.  Now, we note that 
\begin{align*}
\psi_{j,k,k}(t, t)  - \psi_{j,k,k}(t, s) &= \frac{\int_{\R^d} f(x)^{2k+2-j} \dif{x}}{j! ((k+1-j)!)^2} \\
&\times  \int_{(\R^d)^{k+1-j}}  \int_{(\R^d)^{k+1-j}} \int_{(\R^d)^{j-1}}  h_t^k (0,\by_0, \by_1)h_{t,s}^k (0,\by_0,\by_2)\dif \by_0 \dif \by_1 \dif \by_2. 
\end{align*}
Applying a bound $h_t^k(0,\by_0,\by_1) \le \prod_{y\in \by_1} \one \{ \|y\| \le 2T \}$ followed by integrating out $\by_1$, as well as using Lemma \ref{l:hkl_cov} $(i)$, we get
\begin{align*}
\psi_{j,k,k}(t, t)  - \psi_{j,k,k}(t, s) &\le \frac{k^2}{T^d j! \big( (k+1-j)! \big)^2}\, (a_T)^{2k+1-j} (t^d-s^d) \\
&\le \frac{dk^2}{T j! \big( (k+1-j)! \big)^2}\,(a_T)^{2k+1-j} (t-s), 
\end{align*}
where $a_T$ is given in \eqref{e:def.at}. 
Substituting this back into \eqref{e:last.step}, we obtain
\begin{align*}
\E\big[ \big( \mathcal{H}_k(t) - \mathcal{H}_k(s) \big)^{2} \big] &\le \frac{dk^2}{T} \sum_{j=1}^{k+1} \frac{(a_T)^{2k+1-j}}{j! \big( (k+1-j)! \big)^2}\, (t-s) \\
&\le \frac{dk^2}{T(k+1)! a_T}\, \big( a_T (1+a_T) \big)^{k+1} (t-s). 
\end{align*}
Therefore, we conclude that
\begin{align*}
\E\Big[ \big( \mathcal{H}(t) - \mathcal{H}(s) \big)^{2m} \Big] \le \prod_{i=1}^m (2i-1) \Big( \frac{d}{Ta_T} \Big)^m \bigg( \sum_{k=0}^\infty \frac{k\big( a_T(1+a_T)\big)^{(k+1)/2} }{\sqrt{(k+1)!}} \bigg)^{2m} (t-s)^m. 
\end{align*}
One can easily check that the infinite sum on the right hand side converges via the ratio test. As a result, we can apply the Kolmogorov continuity theorem \cite{karatzasshreve}. This implies that there exists a continuous version of $(\mathcal{H}_k(t), \, 0 \leq t \leq T)$ with H{\"o}lder continuous sample paths on $[0,T]$ with any exponent $\gamma \in [0, (m-1)/2m)$. As $m$ is arbitrary, we are done by letting $m\to\infty$. 
\end{proof}

\section{Appendix}

We briefly cite the necessary Palm theory for Poisson processes for an easy reference for the proofs in Section \ref{proofs}. 

\begin{lem}[Lemma 8.1 in \cite{owadaFCLT} and Theorems 1.6, 1.7 in \cite{penr}] \label{l:palm}
Suppose $\Pn$ is a Poisson point process on $\R^d$ with intensity $nf$. Further, for every $k_i \in \Nat_0$, $i=1,\dots,4$, let $h_i(\Y)$ be a real-valued measurable bounded function defined for $\Y\in (\R^d)^{k_i+1}$. By a slight abuse of notation we let $\Y_i$ be collections of $k_i+1$ $\iid$ points with density $f$ on the right hand side of each equation below. We have the following results: \\
$(i)$
$$
\E\Big[ \sum_{\Y_1 \subset \Pn} h_1(\Y_1)\Big] = \frac{n^{k_1+1}}{(k_1+1)!}\, \E[h_1(\Y_1)]. 
$$
$(ii)$ For every $\ell\in \{ 0,\dots,(k_1\wedge k_2)+1 \}$, 
\begin{align*}
&\E \Big[ \sum_{\Y_1\subset \Pn} \sum_{\Y_2\subset \Pn} h_1(\Y_1)h_2(\Y_2)\, \one \big\{ |\Y_1\cap \Y_2|=\ell \big\} \Big] \\
&\quad = \frac{n^{k_1+k_2+2-\ell}}{\ell!(k_1+1-\ell)!(k_2+1-\ell)!}\, \E \big[  h_1(\Y_1)h_2(\Y_2)\, \one \big\{ |\Y_1\cap \Y_2|=\ell \big\} \big]. 
\end{align*}
$(iii)$ For every $\bb = (b_{12}, b_{13}, \dots, b_{1234})\in \Nat_0^{11}$, we have
\begin{align*}
&\E \Big[ \sum_{\Y_1\subset \Pn} \sum_{\Y_2\subset \Pn}\sum_{\Y_3\subset \Pn} h_1(\Y_1)h_2(\Y_2)h_3(\Y_3)\, \\
&\qquad \times \one \big\{ |\Y_1\cap \Y_2|=b_{12}, \, |\Y_1\cap \Y_3|=b_{13}, \, |\Y_2\cap \Y_3|=b_{23}, \, |\Y_1\cap \Y_2 \cap \Y_3|=b_{123}  \big\} \Big] \\
&= \frac{n^{k_1+k_2+k_3+3-b_{12}-b_{13}-b_{23}+b_{123}}}{\prod_{\sigma\subset [3], \, \sigma \neq \emptyset}j_\sigma !}\, \E \big[  h_1(\Y_1)h_2(\Y_2)h_3(\Y_3) \\
&\qquad \times \one \big\{ |\Y_1\cap \Y_2|=b_{12}, \, |\Y_1\cap \Y_3|=b_{13}, \, |\Y_2\cap \Y_3|=b_{23}, \, |\Y_1\cap \Y_2 \cap \Y_3|=b_{123}  \big\} \Big], 
\end{align*}
where 
$$
j_\sigma = \bigg| \bigcap_{i\in \sigma} \Big( \Y_i \setminus \bigcup_{j \in [3]\setminus \sigma}\Y_j \Big)  \bigg|. 
$$
$(iv)$ Furthermore, we have 
\begin{align*}
&\E \Big[ \sum_{\Y_1\subset \Pn} \sum_{\Y_2\subset \Pn}\sum_{\Y_3\subset \Pn}  \sum_{\Y_4\subset \Pn} h_1(\Y_1)h_2(\Y_2)h_3(\Y_3)h_4(\Y_4)\, \\
&\qquad \times \one \big\{ |\Y_i\cap \Y_j|=b_{ij}, \, 1\le i <j \le 4, \ |\Y_i\cap \Y_j \cap \Y_k|=b_{ijk}, \, 1 \le i < j < k\le 4, \\
&\qquad  \qquad \qquad \qquad \qquad \qquad \qquad |\Y_1\cap \Y_2\cap \Y_3 \cap \Y_4|=b_{1234} \big\} \Big] \\
&=\frac{n^{k_1+k_2+k_3+k_4+4-b}}{\prod_{\sigma\subset [4], \, \sigma \neq \emptyset}j_\sigma !}\, \E \Big[ h_1(\Y_1)h_2(\Y_2)h_3(\Y_3)h_4(\Y_4)\, \one \big\{ |\Y_i\cap \Y_j|=b_{ij}, \, 1\le i <j \le 4,  \\
&\qquad \qquad \qquad \quad |\Y_i\cap \Y_j \cap \Y_k|=b_{ijk}, \, 1 \le i < j < k\le 4, \ |\Y_1\cap \Y_2\cap \Y_3 \cap \Y_4|=b_{1234} \big\} \Big], 
\end{align*}
where $b$ and $j_\sigma$ are defined at \eqref{e:def.b} and \eqref{e:def.j.sigma}, respectively.
\end{lem}

%\remove{Let $\X_k = \{X_1, \dots, X_k\}$ where $X_i \in \R^d$ are $\iid$ with density $f$. Further, supppose $\Pn$ a Poisson process with intensity measure $n\int_A f(x) \dif{x}$, for $A \subset \R^d$ Borel. Finally, define $h(\Y)$, $h_i(\Y)$, $i \in \mathbb{N}$ to be real-valued, measurable bounded functions on $(\R^d)^k$. Then, we have the following results: \\ \\
%$(i)$
%\[
%\E\Big[ \sum_{\Y \subset \Pn} h(\Y)\Big] = \frac{n^k}{k!}\E[h(\X_k)],
%\]
%$(ii)$ For $k \in \mathbb{N}$, and $\ell \leq k$ we have
%\begin{align*}
%&\E\Big[ \sum_{\Y_1 \subset \Pn} \cdots \sum_{\Y_k \subset \Pn} \prod_{i=1}^k h_i(\Y_i) \mathcal{I}^k_{[\ell]}(\Y_1, \Y_2, \dots, \Y_k) \Big] \\
%&\qquad \qquad = \E\Big[ \sum_{\Y_1 \subset \Pn} \cdots \sum_{\Y_{k - \ell} \subset \Pn} \prod_{i=1} h_{i+\ell}(\Y_i) \mathcal{I}^k_{\emptyset}(\Y_1, \Y_2, \dots, \Y_{k -\ell}) \Big] \prod_{i=1}^\ell \E\Big[ \sum_{\Y \subset \Pn} h_i(\Y_i) \Big],
%\end{align*}
%with $\mathcal{I}^k_{\sigma}(\Y_1, \Y_2, \dots, \Y_\ell) $ defined at \eqref{e:intersectind}. This simplifies to 
%\[
%\prod_{i=1}^k \E\Big[ \sum_{\Y \subset \Pn} h_i(\Y_i) \Big],
%\]
%in the case $\ell = k$.
%}

\vspace{12pt}
\noindent \textbf{Acknowledgements}: The authors cordially thank the two anonymous referees and the anonymous editor for their insightful and in-depth comments. These suggestions have allowed us to improve the clarity and flow of the paper. TO's research is partially supported by the National Science Foundation (NSF) grant, Division of Mathematical Science (DMS), \#1811428.

%\nocite{*}
\bibliographystyle{APT.bst}
\bibliography{EulerCharacteristicProcess}
\end{document}